%% file: ZGA21.tex
\documentclass[%
  onecolumn
]{mpi2015-cscpreprint}

\usepackage[american]{babel}

\usepackage{graphicx}

\usepackage{amsmath}
\usepackage{amssymb}
\usepackage{amsthm}
\usepackage{algorithmic}
\usepackage{algorithm}
\usepackage{verbatim}
\usepackage{xspace,extarrows,fdsymbol}

\usepackage{tikz}
\usetikzlibrary{positioning}

\usepackage{pgfplots}
\pgfplotsset{compat = 1.13,
	colormap name = viridis,
	unbounded coords = jump}

\tikzset{every picture/.style={/utils/exec={\normalfont}}}

\usepackage{caption}
\usepackage{subcaption}
\definecolor{myRed}{HTML}{E34A33}
\definecolor{myBlue}{HTML}{0571B0}
\definecolor{myBrown}{HTML}{A6611A}

\usepackage{url}
\usepackage[nameinlink]{cleveref}

\crefformat{equation}{\textup{#2(#1)#3}}
\crefrangeformat{equation}{\textup{#3(#1)#4--#5(#2)#6}}
\crefmultiformat{equation}{\textup{#2(#1)#3}}{ and \textup{#2(#1)#3}}
{, \textup{#2(#1)#3}}{, and \textup{#2(#1)#3}}
\crefrangemultiformat{equation}{\textup{#3(#1)#4--#5(#2)#6}}%
{ and \textup{#3(#1)#4--#5(#2)#6}}{, \textup{#3(#1)#4--#5(#2)#6}}
{, and \textup{#3(#1)#4--#5(#2)#6}}

\usepackage{enumitem}

\usepackage{todonotes}

\input{mymacros}

\begin{document}
  

\title{Factorization of the Loewner matrix pencil and its consequences}
  
\author[$\ast$]{Qiang Zhang}
\affil[$\ast$]{Department of Electrical and Computer Engineering, Rice University, Houston, Texas 77005.\authorcr
  \email{qz18@rice.edu}}
  
\author[$\dagger$]{Ion Victor Gosea}
\affil[$\dagger$]{Max Planck Institute for Dynamics of Complex Technical Systems,
	Sandtorstr. 1, 39106 Magdeburg, Germany.\authorcr
  \email{gosea@mpi-magdeburg.mpg.de}, \orcid{0000-0003-3580-4116}}
  
\author[$\dagger$]{Athanasios C. Antoulas}
\affil[$\dagger$]{Department of Electrical and Computer Engineering, Rice University, Houston, Texas 77005, and
	Max-Planck Institute for the Dynamics of Complex Technical Systems,
	D-39106 Magdeburg.\authorcr
	\email{aca@rice.edu}}  
  
\shorttitle{Loewner matrix  pencil factorization and sensitivity}
\shortauthor{Q. Zhang, I. V. Gosea, A. C. Antoulas}
\shortdate{}

\subtitle{\small Dedicated to Professor Volker Mehrmann on the
	occasion of his 66th birthday.}

\keywords{\small Loewner matrix, data-driven modeling, structured perturbation, unstructured perturbation, sensitivity analysis, Sylvester equation. \normalsize}


\abstract{%
This paper starts by deriving a factorization of the Loewner matrix pencil that appears in the data-driven modeling approach known as the Loewner framework and explores its consequences.  The first is that the associated quadruple constructed from the data yields a model without requiring further processing. 
	The second consequence is related to how sensitive the eigenvalues of the Loewner pencil are to perturbations.
	Based on an explicit generalized eigenvalue decomposition of this pencil and by making use of perturbation theory of matrix pencils, we explore two types of eigenvalue sensitivities. The first one is defined with respect to unstructured perturbations of the Loewner pencil, while the second one is defined for structured perturbations. 
	 We also discuss how the choice of data affects the two sensitivities.}

\maketitle


	\section{Introduction}
	
	In many practical applications from science and engineering fields, it is common to model certain physical complex phenomena by means of dynamical systems. In some cases, the number of variables that characterize such systems is very high and does not permit feasible online simulation, nor performing fast controlling tasks. Hence, the need for approximating the original large-scale complex model with a much smaller and simpler model that allows the total computational effort to significantly decrease. Model order reduction (MOR) methods accomplish precisely this goal. Over the years, significant effort has been allocated to refining, optimizing and also including data assimilation in MOR methods. We refer the reader to the following books \cite{Antoulas05,AntBeaGug20,BenOhlCohWill17,QuaManNeg15} for more details on various reduction methodologies. 
	
	The Loewner framework is a data-driven modeling and complexity reduction method that can be used to learn models of dynamical systems from measurements of their transfer function. It was originally introduced in  \cite{Antoulas07} and it was steadily developed over the last decade. For linear systems, extensions of the method were proposed to cope with singular or rectangular systems in \cite{Antoulas17}, with parametric systems in \cite{IonAnt14}, and with preservation of the DAEs (differential algebraic equations) in \cite{Gosea20}.  Moreover, in recent years, several works have been made to extend the Loewner framework to certain classes of nonlinear dynamical systems, such as bilinear systems in \cite{AntGosIon16}, switched systems \cite{GPA18}, and quadratic-bilinear systems in \cite{GosAnt18,AntGosHei19}. For a comprehensive view on the Loewner framework we refer the reader to \cite{Antoulas17,KarGosAnt20}.
	
	One important feature of the Loewner framework consists in the fact that it does not need an exact description of the original dynamical system to start with, which is typically described by ordinary or partial differential equations (ODEs, PDEs). Instead of having full access to the coefficient matrices that scale these equations, one requires only transfer function measurement values. These data can be inferred from time-domain snapshots via spectral transforms  (see \cite{PehGugBea17,KarGosAnt20_2} for details) or directly measured with electronic devices (see \cite{LefAnt10,Antoulas17} for details). Finally, by arranging the given data in a specific way, one can construct with basically no computational effort a realization (dynamical system) that explains the data. The Loewner pencil plays a central role in the system realization constructed by the Loewner framework. More precisely, the two Loewner matrices that enter the pencil represent the coefficient matrices that sc
	ale the internal variable vector and its derivative. Consequently, the eigenvalues of the pencil are the poles of the surrogate Loewner model and are used to characterize the dynamics of the system.

	We first derive an explicit eigenvalue decomposition (EVD) of the Loewner pencil based on a general factorization of the Loewner/shifted Loewner matrices. It was previously shown that the Hankel matrix can be factorized in terms of matrices with special structure, e.g., Vandermonde matrices, in \cite{Sarkar90,Golub99,Golub07}. Similarly, the Loewner matrix can also be factorized in terms of generalized Cauchy matrices, as shown in \cite{Zdenek97}. We will show that the factors given by the generalized Cauchy matrices are actually Krylov projection matrices for a particular system realization. Using the factorization of the generalized Loewner matrix, the EVD of the Loewner pencil is hence available. Based on this EVD and on eigenvalue perturbation theory for matrix pencils, some theoretical aspects of the Loewner pencil perturbation are derived.
	
	To study the sensitivity of eigenvalues in different situations, two kinds of Loewner pencil perturbations are considered. The first one is unstructured perturbation and the perturbation quantities are given by random matrices. 
	The sensitivity $\rho$ with respect to unstructured perturbation is introduced. It is shown that $\rho$ is connected to the condition numbers of the associated generalized Cauchy matrices; $\rho$ is hence an useful tool for sensitivity analysis of the Loewner pencil, e.g., in the computation of pseudospectra \cite{Trefethen05,Embree19}. The second case analyzed in this work considers structured perturbation that usually arises due to noisy data. Because of the special structure of the Loewner matrix, the perturbation matrices of the Loewner pencil are also structured. Some previous works, such as the ones in \cite{Lefteriu10,Zlatko19}, have already studied the effects of noisy measurements in the Loewner model. In this work, we provide a new analysis that takes into consideration the pole sensitivity of the Loewner model with respect to perturbation of data. The sensitivity defined with respect to structured perturbation is denoted with $\eta$. Both sensitivities are influenced by the choice of data.

	The paper is organized as follows; Section \ref{sec:intro_Loew} shows a general factorization of the Loewner pencil and derives an explicit generalized eigenvalue decomposition of the pencil. Next, in Section \ref{sec:Sylvester} we show a number of factorizations for the Loewner pencil, depending on different measurements or on different system realizations. Section \ref{sec:sens} defines sensitivities $\rho$ and $\eta$ with respect to the unstructured and structured perturbations. Furthermore, the consequences resulting from these sensitivities are discussed. Section \ref{sec:numerics} includes numerical examples and discussions to illustrate the consequences of the sensitivities for the different test cases.  Conclusions are given in Section \ref{sec:conc}.
	
\newpage
	
	\section{The Loewner matrix pencil and its properties}
	\label{sec:intro_Loew}
	
	We consider the linear time-invariant dynamical system $\Si$ described by the following equations:
	\begin{equation}\label{equ:sys_def}
	\Si: ~~
	\bE \dot{\bx}(t) = \bA \bx(t) + \bB \bu(t),
	\quad \by(t) = \bC \bx(t), 
	\end{equation}
	where
	$ 
	\bC\in\IR^{p\times n}$, $\bE$,
	$\bA\in\IR^{n\times n}$, $\bB\in\IR^{n\times m}$,
	is a minimal realization of $\Si$. Additionally, $\bx(t) \in \mathbb{R}^n$ is the vector of internal variables, $\bu(t) \in \mathbb{R}^m$ is the input and $\by(t) \in \mathbb{R}^{p \times 1}$ is the output.
	Let the associated resolvent of pencil $(\bA,\bE)$ be $\bPhi(s)=(s\bE-\bA)^{-1}$, and the
	transfer function of $\Si$ be $\bH(s)= \bC(s\bE-\bA)^{-1}\bB = \bC\bPhi(s)\bB$.
	Given the interpolation conditions
	\begin{equation}\label{equ:interp_cond}
	\bell^T_i\bH(\mu_i)=\bv_i^T,~~i=1,\cdots,q,\quad\mbox{and}\quad
	\bw_j=\bH(\lambda_j)\br_j,~~j=1,\ldots,k,
	\end{equation}
	we refer to the {\textit left data} $(\bM,\bL^T,\IV)$ and the {\textit right data}
	$(\bLambda,\bR,\IW)$,  where 
	\begin{align}
	\bM &= \text{diag}\left(\mu_1,\cdots,\mu_q \right)
	\in\mathbb{C}^{q \times q},~
	\bL = \left[ \begin{matrix}
	\bell_1 & \cdots & \bell_q
	\end{matrix} \right]^T \in \mathbb{C}^{p \times q},~
	\IV = \left[ \begin{matrix}\bv_1  & \cdots & \bv_q
	\end{matrix} \right]^T\in \mathbb{C}^{m\times q}, \\[1mm]
	\bLambda &= \text{diag}\left(\lambda_1,\cdots,\lambda_k \right)
	\in \mathbb{C}^{k\times k},~
	\bR = \left[ \begin{matrix}\br_1  & \cdots & \br_k
	\end{matrix} \right] \in \mathbb{C}^{m \times k},~
	\IW = \left[ \begin{matrix}
	\bw_1  & \cdots & \bw_k
	\end{matrix} \right]\in \mathbb{C}^{p\times k}.
	\end{align}
	It is assumed here that the interpolation points 
	$\mu_i$, $\lambda_j$ are mutually distinct.
	The Loewner framework in \cite{Antoulas07} offers a simple solution to the 
	problem of constructing a data-based surrogate linear model with the same structure as in (\ref{equ:sys_def}), that satisfies the conditions in (\ref{equ:interp_cond}). The answer is given by the Loewner quadruple:
	\begin{equation}\label{equ:Loew_quad_def}
	(\IW,\IL,\sIL,\IV)\in\IC^{p\times k}\times
	\IC^{q\times k}\times\IC^{q\times k}\times\IC^{q\times m},
	\end{equation}
	where the Loewner matrix $\IL \in \mathbb{C}^{q\times k}$ and the shifted Loewner matrix  $\sIL \in \mathbb{C}^{	q\times k}$ are defined as:
	\begin{align}\label{equ:Loew_mat_def}
	\left(\IL\right)_{i,j}=\frac{
		\bv_i^T\br_j-\bell_i^T\bw_j}{\mu_i-\lambda_j},~~
	\left(\sIL\right)_{i,j}=\frac{
		\mu_i\bv_i^T\br_j-\bell_i^T\bw_j\lambda_j}{\mu_i-\lambda_j}.
	\end{align}
	The Loewner quadruple is often referred to as the {\textit raw model} of the
	data.
	

	Let $\cK_{{L}}$ and $\cK_{{R}}$ be the associated left/right tangential rational Krylov projection matrices:
	\begin{equation}\label{equ:KL_KR_def}
	\cK_{{L}}=\left[\begin{array}{c}
	\bell_1^T\bC\bPhi(\mu_1)\\[1mm]
	~~~\vdots\\[1mm]
	\bell_q^T\bC\bPhi(\mu_q)\\[1mm]
	\end{array}
	\right]\in\IC^{q\times n},~~\cK_{{R}}=\left[\begin{array}{ccc}
	\bPhi(\lambda_1)\bB\br_1 & \cdots & \bPhi(\lambda_k)\bB\br_k
	\end{array}\right]\in\IC^{n\times k},
	\end{equation}
	assumed to satisfy the condition that $\cK_{{L}} \cK_{{R}}$ has full rank.
	In what follows, we will make use of the {\textit Moore-Penrose generalized-inverse}\footnote{Given a matrix $\bM\in\IC^{k\times \ell}$, its Moore-Penrose 
		generalized inverse denoted by $\bM^+\in\IC^{\ell\times k}$, is the unique
		matrix satisfying the conditions (a) $\bM\bM^+\bM=\bM$, (b) 
		$\bM^+\bM\bM^+=\bM^+$, (c) $(\bM\bM^+)^*=\bM\bM^+$, (d) $(\bM^+\bM)^*=\bM^+\bM$.
		For further, details see \cite{StewartSun}}. It is to be noted that the Drazin inverse could also be used, as shown in \cite{Antoulas16}.

	\begin{lemma} \label{lemma:1}
		The following factorizations hold:
		\begin{equation}\label{factorization}
		\IW= \bC{\cK_{{R}}} \in\IC^{p\times k},~
		\IL= - \cK_{L} \bE \cK_{R} \in\IC^{q\times k},~
		\sIL=- \cK_{L} \bA \cK_{R} \in\IC^{q\times k},~
		\IV= \cK_{L} \bB \in\IC^{q\times m} .
		\end{equation}
		Consequently:\\[-2mm]
		\begin{enumerate}
			\item
			This factorization is rank revealing and
			the rank of $\IL$ is equal to the McMillan degree $n$ of $\Si$.
			The entries of the Loewner quadruple $(\IW,\IL,\sIL,\IV)$, 
			depend exclusively on values of the transfer 
			function $\bH$.	Furthermore the following holds:
			\begin{equation}\label{geninverse}
			\bH(s) = \IW\, (\sIL-s\,\IL)^{+}\, \IV
			\end{equation}
			where $\,(\cdot)^+$ denotes the Moore-Penrose generalized inverse of
			$\,(\cdot)$.	
			\item
			An explicit generalized EVD (eigenvalue decomposition) of the Loewner pencil $(\sIL,\IL)$ results. Let $(\lambda ,\hat\bq,\hat\bp)$ be 
			a triple composed of an eigenvalue, and the right/left eigenvectors 
			of $(\bA,\bE)$. Then $\lambda$ is also an eigenvalue of the pencil
			$(\sIL,\,\IL)$, with corresponding right/left eigenvectors:
			\begin{equation}\label{equ:lr_eigv}
			\bq=\cK_{{R}}^+\hat\bq\quad\mbox{and}
			\quad\bp^T\!=\hat\bp^T\cK_{{L}}^+.
			\end{equation}
		\end{enumerate}
	\end{lemma}
	
	\vspace*{2mm}
	\noindent
	It should be stressed that the above results hold
	irrespective of whether $\sIL-s\,\IL$ is singular or rectangular.
	Therefore the Loewner quadruple is a model of the data and
	there is no need for an explicit projection.
	
	\vspace*{1mm}
	\begin{proof}
		Based on the definition of matrices $\cK_{{L}}$ and $\cK_{{R}}$ provided in (\ref{equ:KL_KR_def}), it follows that the $(i,j)$ entry of matrix ${\cK_{{L}}}\bE{\cK_{{R}}}$ can be written as (for all $1 \leq i \leq q, \ 1 \leq j \leq k$):
		$$\small
		\begin{array}{ll}
		\left({\cK_{{L}}}\bE{\cK_{{R}}}\right)_{i,j} &= \bell_i^T\bC\bPhi(\mu_i) \bE \bPhi(\lambda_j)\bB\br_j = \bell_i^T\bC\bPhi(\mu_i) \frac{\bPhi^{-1}(\lambda_j) -\bPhi^{-1}(\mu_i)}{\lambda_j-\mu_i} \bPhi(\lambda_j)\bB\br_j \\[2mm]
		&= \frac{1}{\lambda_j-\mu_i} \Big{(} \underbrace{\bell_i^T\bC\bPhi(\mu_i) \bB}_{\bv_i^T}\br_j - \bell_i^T \underbrace{\bC \bPhi(\lambda_j)\bB\br_j}_{\bw_j} \Big{)} = \frac{\bv_i^T \br_j - \bell_i^T \bw_j }{\lambda_j-\mu_i}.
		\vspace{-3mm}
		\end{array}
		$$
		Note that in the derivations above, the following identity was used $\bE = \frac{\bPhi^{-1}(\mu_i)-\bPhi^{-1}(\lambda_j)}{\mu_i-\lambda_j}$. From (\ref{equ:Loew_mat_def}) it indeed follows that $\IL = - {\cK_{{L}}}\bE{\cK_{{R}}}$. Similarly, we prove the other identities. These computations 
		can also be found in \cite{Antoulas07}.\\[-2mm]
		
		To show (\ref{geninverse}), we notice that because $\cK_L$ is full column rank
		and $\cK_R$ is full row rank, the generalized inverse of $(\sIL-s\,\IL)$ is:
		$$
		\left[\cK_L(\bA-s\bE)\cK_R\right]^+=\cK_R^T\left(\cK_R\cK_R^T\right)^{-1}
		(\bA-s\bE)^{-1}\left(\cK_L^T\cK_L\right)^{-1}\cK_L^T.
		$$
		The desired result follows by multiplying this expression
		on the left by $\bC\cK_R$ and on the right by $\cK_R\bB$, i.e.
		\begin{align*}
		\IW\, (\sIL-s\,\IL)^{+}\, \IV &= \bC\cK_R \left[\cK_L(\bA-s\bE)\cK_R\right]^+ \cK_R\bB \\
		&= \bC\cK_R\cK_R^T\left(\cK_R\cK_R^T\right)^{-1}
		(\bA-s\bE)^{-1}\left(\cK_L^T\cK_L\right)^{-1}\cK_L^T \cK_R\bB =  \bC  (\bA-s\bE)^{-1} \bB = \bH(s).
		\end{align*}
		
		To show (\ref{equ:lr_eigv}) notice that if 
		$(\lambda ,\hat\bq,\hat\bp)$ is a triple composed of the eigenvalue, and the right/left eigenvectors of $(\bA,\bE)$, $\bq=\cK_{{R}}^+\hat\bq$, 
		is the right eigenvector 
		of the Loewner pencil $(\sIL,\IL)$, corresponding to the eigenvalue $\lambda$:
		\begin{align*}
		{\sIL\bq}= -\cK_{{L}}\bA\cK_{{R}}\bq=-\cK_{{L}}\bA\cK_{{R}}\cK_{{R}}^+\hat\bq=-\cK_{{L}}\bA\hat\bq  =-\lambda \cK_{{L}}\bE\hat\bq =\lambda  \left(-\cK_{{L}}\bE\cK_{{R}}\right)\cK_{{R}}^+\hat\bq{=\lambda\, \IL\bq}.
		\end{align*}
		The proof for the left eigenvector $\bp$ follows in a similar way. 
	\end{proof}
	
	\begin{remark}{
			$\bullet$ If in (\ref{equ:KL_KR_def}), $q$ and $k$ are less that $n$, it is readily shown
			that the transfer function of the projected system (\ref{factorization}),
			interpolates the transfer function of the original system at 
			$\mu_i$, $i=1,\cdots,q$, and $\lambda_j$, $j=1,\cdots,k$ (see e.g. chapter 11
			in \cite{Antoulas05}). In the general case $q,\,k\geq n$, expression 
			(\ref{geninverse}) shows that if more than necessary interpolation conditions hold, then the original system is recovered. From a practical point of view, it follows
			that the emergence of the Loewner matrix in (\ref{factorization}) ensures
			that its rank (or numerical rank) provides an estimate of the complexity
			of the underlying system. Hence by establishing a connection between 
			interpolatory projections and the Loewner matrix as in (\ref{factorization}), 
			we obtain the {\textit additional advantage} that the rank of $\IL$ (exact or
			numerical) provides the complexity of the ensuing reduced models. In addition,
			the projected matrices $\bA$ and $\bE$ (i.e. the Loewner pencil) provide information about the poles of the reduced system. This fact will be used in the next sections to examine the sensitivity of models
			to the choice of data and to perturbation in the data.
		}
	\end{remark}
	
	\section{The generalized Sylvester equation and its impact on factorizing the Loewner pencil} 
	\label{sec:Sylvester}
	
	Given the minimal realization $(\bC,\bE,\bA,\bB)$ $\in$ 
	$\IR^{p\times n}\times\IR^{n\times n}\times\IR^{n\times n}
	\times\IR^{n\times m}$, we consider the interpolatory projection matrices 
	$\cK_L\in\IC^{k\times n}$ and $\cK_R\in\IC^{n\times q}$, $k,\,q\geq n$, 
	such that each has full rank and their product $\cK_L \cK_R\in\IC^{k\times q}$ is non-singular. This means that $\bTheta = \cK_L(\cK_R\cK_L)^{+}\cK_R$
	is a projector, i.e., $\bTheta^2 = \bTheta$. The projected quantities are given by (\ref{factorization}).
	The projection matrices $\cK_L$ and $\cK_R$ satisfy the following {\textit generalized Sylvester equations}:
	\small
	\begin{equation}\label{fundamental}
	\fbox{$\displaystyle
		\begin{array}{c|c}
		\underbrace{{\color{black}\bM}{\color{black}\cK_L}{\color{black}\bE}-
			{\color{black}\bPsi\,}{\color{black}\cK_L}{\color{black}\bA}={\color{black}\bL^T}{\color{black}\bC}}_{
			\begin{array}{ll}
			&\\[-2mm]
			\color{black}(\bC,\bE,\bA):&\color{black}\mbox{ observable triple (w.r.t system)}\\[1mm]
			\color{black}(\bL^T\!,\bPsi,\bM):&\color{black}\mbox{ controllable triple (w.r.t data)}\end{array}}&
		\underbrace{{\color{black}\bE}\,{\color{black}\cK_R}{\color{black}\bLambda}-{\color{black}\bA}\,{\color{black}\cK_R}{\color{black}\,\bDelta}=
			{\color{black}\bB}\,{\color{black}\bR}}_{
			\begin{array}{l}
			\\[-2mm]
			\color{black}(\bE,\bA,\bB):~\,\mbox{ controllable triple (w.r.t system)}\\[1mm]
			\color{black}(\bR,\bLambda,\bDelta):~\mbox{ observable triple (w.r.t data)}\end{array}}
		\end{array}
		$}
	\end{equation}
	\normalsize

	\subsection{Deriving factorizations for different data choices}

	Next we list the projection matrices for some special cases of the data measurements. For simplicity, we treat only the SISO case but all the derivations can be extended to MIMO by incorporating the left and right tangential directions as described in the previous section.

	\subsubsection{Matching at finite values}
	
	In this case, let $\bLambda,\bM$ be diagonal matrices,
	with mutually distinct interpolation points and $\bDelta=\bI_q, \bPsi=\bI_k$, where $\bI_n \in \IR^{n \times n}$ is the identity matrix of dimension $n$. Additionally, let $\bL=\II_q^T$, $\bR=\II_k^T$, where $\II_n = \left[ \begin{matrix}
	1 & 1 & \cdots & 1
	\end{matrix} \right]^T \in \IR^n $ is a $n$th dimensional column vector of ones.
	
	Recall that the system resolvent is denoted by $\bPhi(s)=(s\bE-\bA)^{-1}$. Then, this corresponds to the case presented in the previous section (the classical scenario encountered in the Loewner framework). It follows that the following formulas hold true
	$$
	\cK_L=\left[\begin{array}{c}
	\bC\bPhi(\mu_1)\\\vdots\\\bC\bPhi(\mu_q)
	\end{array}\right],~~
	\cK_R=\left[\begin{array}{ccc}
	\bPhi(\lambda_1)\bB&\cdots&\bPhi(\lambda_k)\bB
	\end{array}\right].
	$$
	
	\noindent
	\textbf{Matching at equal left and right points equal to minus the poles} \\
	
	Next, analyze the case for which the left and right interpolation points are equal to each other, and equal to the mirrored poles of the system (w.r.t to the imaginary axis). Denote with $\pi_1,\ldots,\pi_n$ the poles of the underlying model, and let $\bPi \in \IC^{n \times n}$ be the diagonal matrix containing the poles on its diagonal.
	
	Let $k = q =n$ and $\blambda_i=\bmu_i=-\bpi_i$, for all $1 \leq i \leq n$. This represents a special case encountered in optimal $\cH_2$ approximation, as illustrated in \cite{AntBeaGug20}. Furthermore, choose $\bLambda=\bM=-\bPi$, $\bDelta=\bI_q, \bPsi=\bI_k$, and $\bR=\bL = \II_n^T$. Then write
	$$
	\cK_L=\left[\begin{array}{c}\bC\bPhi(-\pi_1)\\\vdots\\\bC\bPhi(-\pi_n)\end{array}\right],~\cK_R=\left[\bPhi(-\pi_1)\bB~\cdots~\bPhi(-\pi_n)\bB\right],
	$$
	while the Loewner matrices $\IL= - \cK_{L} \bE \cK_{R}$ and $\sIL= - \cK_{L} \bA \cK_{R}$ satisfy the following (degenerate) Sylvester equations:
	$$
	\left\{
	\begin{array}{ccc}
	\IL\bPi-\bPi\IL&=&-\bR^T\bW+\bW^T\bR,\\[2mm]
	\sIL\bPi-\bPi\sIL&=&\bR^T\bW\bPi-\bPi\bW^T\bR.
	\end{array}\right.
	$$
	It is to be noted that in this case, the $(i,i)$ entries of the Loewner matrices are written in terms of the transfer function derivatives, as given below
	\begin{equation*}
	\left(\IL\right)_{i,i}=
	-\frac{d \left[\bH(s) \right]}{d s} \Big{\vert}_{s = -\pi_i}, \ \ \left(\sIL\right)_{i,i}=
	-\frac{d \left[ s \bH(s) \right]}{d s} \Big{\vert}_{s = -\pi_i}.
	\end{equation*}
	
	\noindent
	\textbf{Matching at the same finite point} \\
	
	Let $\lambda \in \IC$ be a complex scalar, and assume in this case that all interpolation points are equal to $\lambda$, i.e., $\mu_i = \lambda$ and $\lambda_j = \lambda$. We presents the results for this through an illustrative simple example. More concrete, choose $k=q=3$, and hence we have that the left quantities are written as follows
	$$
	\bM=\left[\begin{array}{rrr}\lambda&0&0\\-1&\lambda&0\\0&-1&\lambda
	\end{array}\right],~\bL^T=\left[\begin{array}{c}1\\0\\0\end{array}\right]~
	\Rightarrow~\cK_L=\left[\begin{array}{c}
	\bC\bPhi(\lambda)\\\bC\bPhi(\lambda)\bE\bPhi(\lambda)\\
	\bC\bPhi(\lambda)\bE\bPhi(\lambda)\bE\bPhi(\lambda)\end{array}\right],
	$$
	while the right quantities are given below:
	$$
	\bLambda=\left[\begin{array}{rrr}\lambda&-1&0\\0&\lambda&-1\\0&0&\lambda
	\end{array}\right],~\bR^T=\left[\begin{array}{c}1\\0\\0\end{array}\right]~
	\Rightarrow~\cK_R=\left[\begin{array}{ccc}
	\bPhi(\lambda)\bB&\bPhi(\lambda)\bE\bPhi(\lambda)\bB&
	\bPhi(\lambda)\bE\bPhi(\lambda)\bE\bPhi(\lambda)\bB\end{array}\right].
	$$
	Thus, the Loewner matrices $\IL=\cK_L\bE\cK_R$ and $\sIL=\cK_L\bA\cK_R$, become:
	$$\small
	\IL=\left[\begin{array}{ccc}
	\frac{1}{1!}\bH^{(1)}(\lambda)&\frac{1}{2!}\bH^{(2)}(\lambda)&\frac{1}{3!}\bH^{(3)}(\lambda)\\[1mm]
	\frac{1}{2!}\bH^{(2)}(\lambda)&\frac{1}{3!}\bH^{(3)}(\lambda)&\frac{1}{4!}\bH^{(4)}(\lambda)\\[1mm]
	\frac{1}{3!}\bH^{(3)}(\lambda)&\frac{1}{4!}\bH^{(4)}(\lambda)&\frac{1}{5!}\bH^{(5)}(\lambda)\\
	\end{array}\right],~\sIL=\lambda\IL\,+
	\left[\begin{array}{ccc}
	\bH^{(0)}(\lambda)&\frac{1}{1!}\bH^{(1)}(\lambda)&\frac{1}{2!}\bH^{(2)}(\lambda)\\[1mm]
	\frac{1}{1!}\bH^{(1)}(\lambda)&\frac{1}{2!}\bH^{(2)}(\lambda)&\frac{1}{3!}\bH^{(3)}(\lambda)\\[1mm]
	\frac{1}{2!}\bH^{(2)}(\lambda)&\frac{1}{3!}\bH^{(3)}(\lambda)&\frac{1}{4!}\bH^{(4)}(\lambda)\\
	\end{array}\right].
	$$
	It follows that the equations satisfied by the Loewner and the shifted Loewner matrices can be rewritten as:
	$$\small
	\begin{array}{l}
	\IL\bJ_r-\bJ_r^T\IL=\left[\begin{array}{c}1\\0\\0\end{array}\right]
	\left[\bH^{(0)}(\lambda) ~\frac{1}{1!}\bH^{(1)}(\lambda) ~
	\frac{1}{2!}\bH^{(2)}(\lambda)\right]-\left[
	\begin{array}{c}\bH^{(0)}(\lambda)\\[1mm]\frac{1}{1!}\bH^{(1)}(\lambda)\\[1mm]
	\frac{1}{2!}\bH^{(2)}(\lambda)\end{array}\right][1 ~0 ~0],
	\\[3mm]
	(\sIL-\lambda\IL)\bJ_r-\bJ_r^T(\sIL-\lambda\IL)=
	\left[
	\begin{array}{c}0\\[1mm]\bH^{(0)}(\lambda)\\[1mm]\frac{1}{1!}\bH^{(1)}(\lambda)
	\\[1mm]
	\end{array}\right][1 ~0 ~0]-
	\left[\begin{array}{c}1\\0\\0\end{array}\right]
	\left[0 ~\bH^{(0)}(\lambda) ~\frac{1}{1!}\bH^{(1)}(\lambda)\right].
	\\[1mm]
	\end{array}
	$$
	\normalsize

	
	\subsubsection{Matching at infinity}
	
	Let $\bM=\bI_q$, $\bL=\bfe_{1,q}^T$, $\bPsi\,=\bJ_q^T$, and also
	$\bLambda=\bI_k$, $\bR=\bfe_{1,k}^T$, $\,\bDelta=\bJ_k$. Here $\bJ_n \in \IR^{n \times n}$ is a Jordan matrix of dimension $n$ with 0 eigenvalues and $\bfe_{1,n} \in \IR^n$ is the first unit vector of length $n$ (we sometime use $\bfe_{1}$ for ease of notation). In this case, we also assume that $\bE=\bI_n$. The following derivations hold
	\small
	$$
	\left.\begin{array}{l}
	\cK_L-\bJ_q^T\cK_L\bA=\bfe_1\bC\\[1mm]
	\cK_R-\bA\cK_R\bJ_k=\bB\bfe_1^T\end{array}\right\}\Rightarrow
	\cK_L=\left[\!\begin{array}{c}\bC\\\bC\bA\\\vdots\\\bC\bA^{q-1}
	\end{array}\!\!\right],~\cK_R=[\bB,~\bA\bB,~\cdots,\,\bA^{k-1}\bB]~
	\Rightarrow~\left\{
	\begin{array}{l}
	\IL=\cK_L\cK_R=\cH_{q,k},\\[1mm]\sIL=\cK_L\bA\cK_R=\sigma\cH_{q,k}.
	\end{array}\right.
	$$
	\normalsize
	Here,  $\cH_{q,k} \in \IR^{q \times k}$ and  $\sigma\cH_{q,k} \in \IR^{q \times k}$ are Hankel matrices that contain as entries the Markov parameters $\bh_i = \bC \bA^{i-1} \bB, i \geq 1$ of the original system, i.e. $ \left(\cH_{q,k} \right)_{i,j} = \bh_{i+j-1}$ and $ \left(\sigma \cH_{q,k} \right)_{i,j} = \bh_{i+j}$ for $i.j \geq 1$.
	$$
	\left.\begin{array}{l}
	\cK_L{\cK_R}-\bJ_q^T\cK_L\bA{\cK_R}=
	\bfe_1\bC{\cK_R}\\[2mm]
	{\cK_L}\cK_R-{\cK_L}\bA\cK_R\bJ_k=
	{\cK_R}\bB\bfe_1^T\end{array}\right\}\Rightarrow\left\{
	\begin{array}{c}
	\cH_{q,k}-\bJ_q^T\sigma\cH_{q,k}=\left[\begin{array}{c}
	\bh_1,~\bh_2,~\cdots,~\bh_k\\
	\bfz\\\vdots\end{array}\right]=\bfe_1\,\bfe_1^T \cH_{q,k},\\[5mm]
	\cH_{q,k}-\sigma\cH_{q,k}\bJ_k=\left[
	\begin{array}{ccc}
	\bh_1&0&\cdots\\\bh_2&0&\cdots\\\vdots&\vdots&\ddots\\\bh_q&0&\cdots\end{array}\right]=\cH_{q,k} \bfe_1\,\bfe_1^T.
	\end{array}\right.
	$$
	The equations for the Hankel and shifted Hankel matrices are given below:
	$$
	\begin{array}{l}
	\cH_{q,k}\bJ_k-\bJ_q^T\cH_{q,k}=
	\bfe_1 \bfe_1^T\cH_{q,k}\bJ_k-\bJ_q^T\cH_{q,k} \bfe_1 \bfe_1^T,~~\\[2mm]
	\sigma\cH_{q,k}\bJ_k-\bJ_q^T\sigma\cH_{q,k}=\bfe_1 \bfe_1^T \cH_{q,k}-
	\cH_{q,k} \bfe_1 \bfe_1^T .
	\end{array}
	$$

	\subsection{Deriving factorizations for different system realizations}
	
	Next we will derive several factorizations of the Loewner quadruple based
	on different system realizations and data. For clarity
	we will discuss three separate special cases: (a) SISO 
	(single-input single-output) systems with strictly proper rational (spr) 
	transfer functions, (b) MIMO (multiple-input multiple output) systems with
	spr transfer functions and finally (c) systems with polynomial transfer 
	functions. The general case of systems characterized by DAEs (differential algebraic equations),
	with arbitrary rational transfer functions follows readily as a combination
	of these special cases.
	
	\subsubsection{The case of SISO systems with strictly proper transfer function}
	
	Let an underlying linear SISO ($m=p=1$) system $\Si$ of dimension $n$ as defined in (\ref{equ:sys_def}) be 
	represented by means of its partial fraction decomposition:
	\begin{equation}\label{equ:zero-pole-siso}
	\bH(s)=\frac{\bn(s)}{\bd(s)}=\sum_{i=1}^n\frac{\gamma_i}{s-\pi_i},
	\end{equation}
	where $\gamma_i \neq 0 \in\IC$, for all $i$. The poles and the residues of the system are denoted by $\{\pi_1,\cdots,\pi_n\}$, and by $\{\gamma_1,\cdots,\gamma_n\}$, respectively. Let diagonal matrices be defined as $\bPi=\mbox{diag}[\pi_1,\cdots,\pi_n]\in\IC^{n\times n}$ and
	$\bGamma=\mbox{diag}[\gamma_1,\cdots,\gamma_n]\in\IC^{n\times n}$, while $ \bgamma = [\gamma_1,\cdots,\gamma_n]^T\in\IC^{n\times 1}$.
	Then, a realization for its transfer function $\bH(s)$, defined in (\ref{equ:zero-pole-siso}), is given by:
	\begin{equation}\label{equ:realiz_pol_res}
	\bA = \bPi\bGamma, ~ \bE = \bGamma,~ \bB =  \bGamma{\II }_n, ~ \bC = {\II }_n^T \bGamma.
	\end{equation}
	where $\II_n$ is all-ones column vector of length $n$. It is to be noted that an equivalent realization to that in (\ref{equ:realiz_pol_res}) is given by: $\bA = \bPi, ~ \bE = \bI_n,~ \bB =  {\II }_n, ~ \bC = {\II }_n^T \bGamma$.

	To the SISO system $\Si$, we associate a Loewner quadruple (raw model)
	\begin{equation}\label{equ:Loew_quad_def2}
	(\IW,\IL,\sIL,\IV)\in\IC^{1\times k}\times
	\IC^{q\times k}\times\IC^{q\times k}\times\IC^{q\times 1},
	\end{equation}
	constructed by means of the left interpolation points $\mu_i$, $i=1,\cdots,q$,
	the right interpolation points $\lambda_j$, $j=1,\cdots,k$ ($\lambda_j$, $\mu_i$ assumed pairwise distinct), and the left values $\bv_i=\bH(\mu_i)$ and the right values $\bw_j=\bH(\lambda_j)$, where $q,k\geq n$. Additionally, by choosing $\bR={\II }_q^T$, $\bL={\II }^T_k$, it follows that
	\begin{align}\label{equ:Loew_mat_def2}
	\IW_j = \bw_j,~~\left(\IL\right)_{i,j}=\frac{
		\bv_i-\bw_j}{\mu_i-\lambda_j},~~
	\left(\sIL\right)_{i,j}=\frac{
		\mu_i\bv_i-\bw_j\lambda_j}{\mu_i-\lambda_j},~~\IV_i = \bv_i.
	\end{align}
	Given two sets of mutually
	distinct complex numbers $\alpha_i$, $i=1,\ldots,\kappa$, and 
	$\beta_j$, $j=1,\ldots,\rho$,
	we define the associated {\textit Cauchy matrix} as:
	\begin{equation}\label{equ:Cauchy_def}
	\cC_{\alpha,\beta}=\left[\begin{array}{ccc}
	\frac{1}{\alpha_1-\beta_1}& \cdots & \frac{1}{\alpha_1-\beta_\rho}\\[1mm]
	\vdots & \cdots & \vdots \\[1mm]
	\frac{1}{\alpha_\kappa-\beta_1}& \cdots & \frac{1}{\alpha_\kappa-\beta_\rho}\\[1mm]
	\end{array}\right]\in\IC^{\kappa\times \rho}.
	\end{equation}
	
	\begin{lemma} \label{lemma:2}
		Given the Loewner quadruple (\ref{equ:Loew_quad_def2}), 
		the following factorizations hold true:
		\begin{equation}\label{equ:facsiso}
		\IW={\II }_n^T\,\bGamma\,
		\cC^T_{\lambda,\pi} \in\IC^{1\times k},~
		\IL=-{\cC_{\mu,\pi}}\bGamma\,\cC^T_{\lambda,\pi},~
		\sIL=-{\cC_{\mu,\pi}}\bPi\bGamma\cC^T_{\lambda,\pi} \in\IC^{q\times k},~\IV={\cC_{\mu,\pi}}\,\bGamma{\II }_n \in\IC^{q\times 1}.
		\end{equation}
	\end{lemma}
	
	\begin{proof}
		For $1 \leq i \leq q$ and $1 \leq j \leq k$, we can write the $(i,j)$ entry of the Loewner matrix $\IL$ in (\ref{equ:Loew_mat_def2}), as follows
		$$
		\IL_{(i,j)} = \frac{\bv_i-\bw_j}{\mu_i-\lambda_j} = \frac{\bH(\mu_i)-\bH(\lambda_j)}{\mu_i-\lambda_j} = \frac{\Big( 
			\sum_{l=1}^n \frac{\gamma_l}{\mu_j-\pi_l} - \sum_{\l=1}^n \frac{\gamma_l}{\lambda_j-\pi_l}
			\Big)}{\mu_i-\lambda_j}
		= \sum_{l=1}^n  \gamma_l \frac{1}{\mu_i-\pi_l} \frac{1}{\pi_l- \lambda_j}.
		$$
		Hence, it follows that the $(i,j)$ entry of the Loewner matrix $\IL$ coincides with the $(i,j)$ entry of the matrix computed by the product $-{\cC_{\mu,\pi}}~\bGamma~\cC^T_{\lambda,\pi}$, for all $i,j$. Hence, the second equality in (\ref{equ:facsiso}) holds. Similarly, we can show the other three equalities. 
	\end{proof}
	
	\begin{proposition}\label{prop:Cauchy_Sylv}
		The Cauchy matrix in (\ref{equ:Cauchy_def}) satisfies the Sylvester equation $\bD_\alpha \cC_{\alpha,\beta}-\cC_{\alpha,\beta} \bD_\beta={\II }_\kappa{\II }_\rho^T$, where $\bD_\alpha=\mbox{diag}[\alpha_1,\cdots,\alpha_\kappa],~~
		\bD_\beta=\mbox{diag}[\beta_1,\cdots,\beta_\rho]$. Hence, the Cauchy matrices that are used for the factorizations in (\ref{equ:facsiso}) satisfy the Sylvester equations:
		\begin{equation}\label{equ:Cauchy_Sylv}
		\bM \cC_{\mu,\pi}-  \cC_{\mu,\pi} \bPi={\II }_q{\II }_n^T, \ \ \cC^T_{\lambda,\pi}\bLambda-\bPi\cC^T_{\lambda,\pi}={\II }_n{\II }_k^T.
		\end{equation}
	\end{proposition}
	
	It is to be noted that the equations in (\ref{equ:Cauchy_Sylv}) represent special cases of those in (\ref{fundamental}), for a particular realization of the original system.

	\begin{lemma} \label{lemma:11}
		The right and left eigenvectors of the Loewner matrix pencil given by 
		$(\sIL, \IL)$ corresponding to the eigenvalue $\pi_i$ are given, respectively, by
		\begin{align} \label{equ:lr_eigv2}
		\bq_i &= (\mathcal{C}_{\lambda, \pi}^T)^{+} \mathbf{e}_i, \quad \text{and} \quad
		\bp_i = (\mathcal{C}_{\mu, \pi}^T)^{+} \mathbf{e}_i,
		\end{align}
		where $\mathbf{e}_i$ is the unit vector of length $n$ whose $i^{th}$ entry is $1$ (and all others are zeros).
	\end{lemma}
	
	\begin{proof}
		The proof is similar to that used to prove part c). of Lemma \ref{lemma:1}. Since the factorization in (\ref{equ:facsiso}) holds, it is easy to show that $\left(\pi_i, \mathbf{e}_i, \mathbf{e}_i \right)$ for $i = 1,\ldots, n$ are the triple (eigenvalue, left eigenvector, right eigenvector) of pencil $\left(\bPi, \bI \right)$. Then, the identity stated in (\ref{equ:lr_eigv2}) can be obtained by simply substituting $\cK_{{R}}$ with $\mathcal{C}_{\lambda, \pi}^T$ and $\cK_{{L}}$ with $\mathcal{C}_{\mu, \pi}$ into the formulas given in (\ref{equ:lr_eigv}). 
	\end{proof}
	If the left interpolation points are the same as the right interpolation points, then
	\begin{align}\label{equ:lr_eigv3}
	\bq_i =\bp_i= (\mathcal{C}_{\mu, \pi}^T)^{+} \mathbf{e}_i,
	\end{align} 
	which means the left eigenvector coincide to the right eigenvector.
	
	{\textbf Connection with the Hankel singular values.}
	Again consider the realization of $\Si$ as given in (\ref{equ:realiz_pol_res}). Moreover, consider that $\Si$ is a stable system. Then, it follows that the controllability Gramian is $\mathcal{P} = \mathcal{C}_{-\pi, \pi^*}$, while the observability Gramian is given by $\mathcal{Q} = -\bGamma^* \mathcal{C}_{\pi^*, -\pi} \bGamma$. We conclude that if $k = q = n$ and the interpolation points are chosen as $\mu_i = -\pi^*_i$ and
	$\lambda_j = -\pi^*_j$, and the resulting Loewner matrix satisfies:\\[-2mm]
	\begin{center} 
		$ 
		-\IL^* \bGamma = \mathcal{P}\mathcal{Q}.
		$\\[2mm]
	\end{center}
	
	\noindent
	Thus, the eigenvalues of the squares of the Hankel singular values $\sigma_i$, $i = 1, \ldots, n$, of the system
	are equal to the eigenvalues of the resulting Loewner matrix $\IL$ scaled by the diagonal matrix $\bGamma$:
	\begin{equation} \label{equ:hsv}
	\sigma^2_i(\Si) = \sigma(\IL^* \bGamma)_i, ~~~~ i = 1, \ldots, n.
	\end{equation}

	\subsubsection{The case of MIMO systems with strictly proper transfer function}
	\label{sec:MIMO_str_prop}
	
	Let $\bC=\left[\bc_1\cdots\bc_n\right]\in\IR^{p\times n},~~
	\bB^T=\left[\bb_1\cdots\bb_n\right]\in\IR^{m\times n}$ with $\bc_i\in\IR^p,
	~~\bb_i\in\IR^m$ and also assume that $\bA = \bPi$.~ Hence, the transfer function $\bH(s)$ can be expressed in pole-residue form as follows
	\begin{equation}\label{equ:zero-pole-mimo}
	\bH(s)=\bC \ \left(s\bI-\bPi\right)^{-1} \bB =\sum_{i=1}^n\displaystyle\frac{\bc_i \bb_i^T}{s-\pi_i}.
	\end{equation}
	Consider now the generalized Cauchy matrices $\cC_{{L}}\in\IC^{q\times n}$ and $\cC_{{R}} \in\IC^{n\times k}$, defined as follows:
	\begin{align} \label{equ:CL_CR_def}
	(\cC_{{L}})_{i,j}=\frac{\bell_i^T\bc_j}{\mu_i-\pi_j}, \ \ \ (\cC_{{R}})_{i,j}=\frac{\bb_i^T\br_j}{\lambda_j-\pi_i},
	\end{align}
	that can be obtained as solutions of Sylvester equations:
	\begin{align}\label{equ:Cauchy_Gen_Sylv}
	\bM \cC_{{L}}-{\cC_{{L}}\bPi=\bL^T\bC}, \ \ \ \cC_{{R}} \bLambda - \bPi \cC_{{R}} =\bB\bR.
	\end{align}
	As before, it is to be noted that the equations in (\ref{equ:Cauchy_Gen_Sylv}) represent special cases of those in (\ref{fundamental}), for a particular realization of the original system as given in (\ref{equ:zero-pole-mimo}).

	\begin{lemma} \label{lemma:3}
		In the above setting, the relationships below hold true:
		\begin{equation}\label{equ:facmimo}
		\IW= \bC\,\cC_{{R}} \in\IC^{p \times k},~~
		\IL= -\cC_{{L}}\cC_{{R}},~
		\sIL=-\cC_{{L}}\bPi\cC_{{R}} \in\IC^{q\times k},~~
		\IV= \cC_{{L}}\bB \in\IC^{q\times m}.
		\end{equation}
	\end{lemma}
	
	\begin{proof}
		For a pair of indexes $1 \leq i \leq q$ and $1 \leq j \leq k$, we can write the $(i,j)$ entry of the Loewner matrix $\IL$ in (\ref{equ:Loew_mat_def}), as follows:\small
		\begin{align*}
		\IL_{(i,j)} &= \frac{ \bv_i^T\br_j-\bell_i^T\bw_j}{\mu_i-\lambda_j} = \frac{\bell_i^T\bH(\mu_i)\br_j-\bell_i^T\bH(\lambda_j )\br_j}{\mu_i-\lambda_j} = \frac{\Big( 
			\sum_{l=1}^n \frac{\bell_i^T\bc_l \bb_l^T\br_j}{\mu_j-\pi_l} - \sum_{\l=1}^n \frac{\bell_i^T\bc_l \bb_l^T\br_j}{\lambda_j-\pi_l}
			\Big)}{\mu_i-\lambda_j} \\
		&=	\sum_{l=1}^n \bell_i^T\bc_l \bb_l^T\br_j \frac{ (\lambda_j-\pi_l)-(\mu_i-\pi_l) } {(\lambda_j-\pi_l)(\mu_i-\pi_l)(\mu_i-\lambda_j)} = -\sum_{l=1}^n    \frac{\bell_i^T\bc_l}{\mu_i-\pi_l} \frac{\bb_l^T\br_j}{ \lambda_j- \pi_l}.
		\end{align*}\normalsize
		Hence, it follows that the $(i,j)$ entry of the Loewner matrix $\IL$ coincides with the $(i,j)$ entry of the matrix computed by the product $-\cC_{{L}}\cC_{{R}}$, for all $i,j$. Hence, the second equality in (\ref{equ:facmimo}) holds. The others can be proven similarly.
	\end{proof}
	
	\vspace*{2mm}
	\noindent
	One can notice that if all rows of $\bL^T$ are the same, 
	i.e. $\bL^T={\II }_q \bell^T$, then
	$$
	\cC_{{L}}=\cC_{\mu,\pi}
	\underbrace{\mbox{diag}[\bell^T\!\bc_1,\cdots,\bell^T\!\bc_n]}_{\bDelta_\bC},
	~\mbox{where}~\cC_{\mu,\pi} ~\mbox{is a }~q\times n~\mbox{ Cauchy matrix}.
	$$
	Furthermore, if all columns of $\bR$ are the same, i.e. $\bR=\br {\II }_k^T$, then
	$$
	\cC_{{R}}=
	\underbrace{\mbox{diag}[\bb_1^T\br,\cdots,\bb_n^T\br]}_{\bDelta_\bB}
	\cC^T_{\lambda,\pi}, ~\mbox{where} ~ \cC^T_{\lambda,\pi} ~\mbox{is a }~n\times k~\mbox{ Cauchy matrix}.
	$$
	If both of these conditions hold, we have that
\begin{align*}
	\IW&= \bC\bDelta_\bB
	\cC^T_{\lambda,\pi} \in\IC^{p \times k},~~
	\IL= -\cC_{\mu,\pi} \bDelta_\bC \bDelta_\bB\
	\cC^T_{\lambda,\pi},\\
	\sIL&= -\cC_{\mu,\pi} \bDelta_\bC \bPi \bDelta_\bB
	\cC^T_{\lambda,\pi} \in\IC^{q\times k},~~
	\IV= \cC_{\mu,\pi}\bDelta_\bC\bB \in\IC^{q\times m}.
	\end{align*}
	If $m=p=1$ the above relationships reduce to ~{(\ref{equ:facsiso})}.

	\subsubsection{The case of SISO systems with polynomial transfer functions}
	
	In this section we assume that the transfer function is a polynomial, i.e., ~$\bH(s)=a_{r-1}s^{r-1}+\cdots+a_1s+a_0$. For simplicity, we treat only the scalar case. A minimal realization 
	of $\bH$ is given by $\bH(s)=\bC(s\bE-\bA)^{-1}\bB$,~ where:
	\begin{equation}\label{minrealpoly}
	\bE=\bJ_r, ~~\bA=\bI_r,~~\bB=\bfe_r,~~\bC=\left[\begin{matrix}
	a_{r-1} & \cdots & a_1 & a_0 \end{matrix} \right],
	\end{equation}
	$\bJ_r$ is a $r\times r$ Jordan block with zero eigenvalues
	and ones in the superdiagonal and $\bfe_r \in \IR^r$ is 
	the $r^{\textrm th}$ unit vector. 
	To compute the Loewner pencil, 
	we choose the left interpolation points as $\mu_i$, $i=1,\cdots,q$,
	and the right interpolation points as $\lambda_j$, $j=1,\cdots,k$.
	The Loewner pencil of size $q\times k$ is then represented as:
	\begin{equation}\label{loew_poly_form}
	\left\{\begin{array}{ll}
	(\IL)_{i,j}&=\frac{\bH(\mu_i)-\bH(\lambda_j)}{\mu_i-\lambda_j}=
	a_{r-1}(\mu_i^{r-2}+\mu_i^{r-3}\lambda_j+\cdots+\mu_i\lambda_j^{r-3}+\lambda_j^{r-2})+
	\cdots+a_2(\mu_i+\lambda_j)+a_1,\\[3mm]
	{(\sIL)}_{i,j}&=\frac{\mu_i\bH(\mu_i)-\lambda_j\bH(\lambda_j)}{\mu_i-\lambda_j}=
	a_{r-1}(\mu_i^{r-1}+\mu_i^{r-2}\lambda_j+\cdots+\mu_i\lambda_j^{r-2}+\lambda_j^{r-1})+
	\cdots+a_1(\mu_i+\lambda_j)+a_0.
	\end{array}\right.
	\end{equation}
	With $x=[x_1,\cdots,x_m]$, the associated {\textit Vandermonde matrix} is:
	\begin{equation}\label{vandermonde} \footnotesize
	\bV_{\ell,m}(x)=\left[\begin{array}{cccc}
	1&1&\cdots&1\\[1mm]
	x_1&x_2&\cdots&x_m\\[1mm]
	\vdots&\vdots&\ddots&\vdots\\[1mm]
	x_1^{\ell-1}&x_2^{\ell-1}&\cdots&x_m^{\ell-1}
	\end{array}\right]\in\IC^{\ell\times m}.
	\end{equation}
	Then, based on (\ref{loew_poly_form}), and on the above definition, 
	it follows that
	\begin{equation}\label{equ:facLLspoly}
	\left\{\begin{array}{ll}
	\IL&=a_1\bV^T_{1,q}(\mu)\bV_{1,k}(\lambda)+a_2\bV^T_{2,q}(\mu) \bV_{2,k}(\lambda)+
	\cdots+a_{r-1}\bV_{r-1,q}^T(\mu)\bV_{r-1,k}(\lambda),\\[3mm]
	\sIL&=a_0\bV_{1,q}^T(\mu)\bV_{1,k}(\lambda)+a_1\bV_{2,q}^T(\mu)\bV_{2,k}(\lambda)+
	\cdots+a_{r-1}\bV_{r,q}^T(\mu)\bV_{r,k}(\lambda),
	\end{array}\right.
	\end{equation}
	and also, since $\IW_j = \bH(\lambda_j)$ and  $\IV_i = \bH(\mu_i)$, we get that
	\begin{align}\label{equ:facVWpoly} \IW=\left[\begin{array}{cccc}a_0&a_1&\cdots&a_{r-1}\end{array}\right] \bV_{r,k}(\mu) ~~\text{and}~~
	\IV=\bV_{r,q}^T(\lambda)
	\left[\begin{array}{cccc}a_0&a_1&\cdots&a_{r-1}\end{array}\right]^T.
	\end{align}
	
	\begin{lemma} \label{lemma:2poly}
		Given the Loewner quadruple introduced in (\ref{equ:facLLspoly}) 
		and (\ref{equ:facVWpoly}), the following 
		factorizations hold true:
		\begin{align}\label{equ:facpoly}
		\IW &= \hat{\bC} \bV_{r,k}(\lambda) \in\IC^{1\times k},~
		\IL= \bV_{r,q}^T(\lambda) \hat{\bE} \bV_{r,k}(\mu),\\
		\sIL&= \bV_{r,q}^T(\lambda) \hat{\bA} \bV_{r,k}(\mu) \in\IC^{q\times k},~
		\IV= \bV_{r,q}^T(\mu) \hat{\bB}   \in\IC^{q\times 1},
		\end{align}
		where the following notation is used
		\begin{equation}\label{equ:def_hatE_hatA}\footnotesize
		\hat{\bE}\! =\! \left[\!\begin{array}{cccccc}
		a_1 & a_2 &\cdots & a_{r-2}& a_{r-1}&0\\
		a_2 & a_3 & \cdots &a_{r-1} & 0&0\\a_3&a_4&\cdots&0&0&0\\
		\vdots & \vdots&\ddots &\vdots &\vdots &\vdots \\
		a_{r-1} &0&\cdots& 0   & 0&0\\
		0&0&\cdots &0&0&0\\
		\end{array}\!\right]\!,~
		\hat{\bA}\! =\! \left[\!\begin{array}{cccccc}
		a_0 & a_1 &\cdots & a_{r-3}& a_{r-2}&a_{r-1}\\
		a_1 & a_2 & \cdots &a_{r-2} & a_{r-1}&0\\
		a_2&a_3&\cdots&a_{r-1}&0&0\\
		\vdots & \vdots&\ddots &\vdots &\vdots &\vdots \\
		a_{r-2} &a_{r-1}&\cdots& 0   & 0&0\\
		a_{r-1}&0&\cdots &0&0&0\\
		\end{array}\!\right]\!,~
		\hat{\bB}\! =\! \left[\!\begin{array}{c}a_0\\a_1\\\vdots\\a_{r-1}
		\end{array}\!\right]\! = \hat{\bC}^T.
		\end{equation}
	\end{lemma}
	\begin{proof}
		The factorizations in (\ref{equ:facpoly}) directly follow from	the representations in (\ref{equ:facLLspoly}) and in (\ref{equ:facVWpoly}).
	\end{proof}
	
	\section{Sensitivities of the Loewner pencil eigenvalues} \label{sec:sens}
	In this section, we provide definitions for sensitivities of eigenvalues of the Loewner pencil. These concepts of sensitivity arise from the perturbation of eigenvalues of the Loewner matrix pencil. In the following, we introduce the eigenvalue perturbation theory for the matrix pencil. Afterwards, we discuss the perturbation of the Loewner pencil and define two types of sensitivities, denoted with $\rho$ and $\eta$, corresponding to structured or unstructured perturbations.

	\begin{lemma} \label{lemma:5}
		For the matrix pencil $\left(\bA, \bE\right)$ that is assumed to be diagonalizable, under the perturbation of $\bar{\bA} = \bA + \bDelta_\bA$ and $\bar{\bE} = \bE + \bDelta_\bE$, the first order approximation of the eigenvalue perturbation $\pi^{(1)}$ is
		\begin{align*}	
		\pi^{(1)} &= \frac{\bp^T\left( \bDelta_\bA - \pi \bDelta_\bE \right) \bq}{\bp^T \bE \bq},
		\end{align*}
		where the eigenvalue problem of the pencil $(\bA,\bE)$ is: 
		$\bA\bq=\pi\bE\bq$,~ $\bp^T\bA=\pi\bp^T\bE$, 
		and $\pi$ is assumed to be an eigenvalue with algebraic multiplicity equal to 1.
	\end{lemma}
	\begin{proof}
		A perturbation of the system yields
		$$\small
		\begin{array}{ll}
		{(\bA+\bDelta_\bA)}(\bq+\bq^{(1)}+\cdots)&=\quad
		(\pi+\pi^{(1)}+\cdots){(\bE+\bDelta_\bE)}
		(\bq+\bq^{(1)}+\cdots),\\[1mm]
		(\bp^T+\bp^{(1)^T}+\cdots){(\bA+\bDelta_\bA)}&=\quad
		(\pi+\pi^{(1)}+\cdots)(\bp^T+\bp^{(1)^T}+\cdots)
		{(\bE+\bDelta_\bE)}.
		\end{array}
		$$
		By retaining only the first-order terms, it follows that:
		$(\bA-\pi\bE)\bq^{(1)}=(\pi^{(1)}\bE+\pi\bDelta_\bE-\bDelta_\bA)\bq$.
		We wish to find an expression for $\pi^{(1)}$, which measures the first-order 
		sensitivity of $\pi$. Towards this goal we multiply this equation 
		on the left by the corresponding left eigenvector $\bp$:
		$$\small
		\begin{array}{c}
		{ \bp^T}(\bA-\pi\bE)\bq^{(1)}=
		{ \bp^T}(\pi^{(1)}\bE+\pi\Delta_\bE-\Delta_\bA)\bq~\Rightarrow
		\underbrace{{ \bp^T}\bA\bq^{(1)}-\pi{ \bp^T}\bE\bq^{(1)}}_{=0}=
		\pi^{(1)}{ \bp^T}\bE\bq+\pi{ \bp^T}\Delta_\bE\bq-
		{ \bp^T}\Delta_\bA\bq\\[6mm]~\Rightarrow
		\pi^{(1)}=\frac{\displaystyle
			{ \bp^T}\Delta_\bA\bq-\pi{ \bp^T}\Delta_\bE\bq}{\displaystyle{ \bp^T}\bE\bq}=\frac{
			\displaystyle{ \bp^T}(\Delta_\bA-\pi\Delta_\bE)\bq}
		{\displaystyle{ \bp^T}\bE\bq} .
		\end{array}
		$$	
	\end{proof}
	
	Now, consider the Loewner pencil $\left(\sIL, \IL\right)$ as a surrogate to the pencil $(\bA,\bE)$. The perturbed matrices are as follows, i.e., $\bar{\IL} = \IL + \bDelta_{L}$, ~ $\bar{\IL}_s = \sIL + \bDelta_{L_s}$. Then, by following the result in Lemma \ref{lemma:5}, we have that the first order approximation of the eigenvalue perturbation $\pi^{(1)}$ is
	\begin{equation}	\label{equ:sensitivity}
	\pi^{(1)} = \frac{\bp^T\left( \bDelta_{L_s} - \pi \bDelta_{L} \right) \bq}{\bp^T \IL \bq}.
	\end{equation}
	Note that in (\ref{equ:sensitivity}), the left/right eigenvectors denoted with 
	$\bp$, $\bq$, can be actually obtained as in (\ref{equ:lr_eigv}), or as in (\ref{equ:lr_eigv2}). It can hence be noticed that the first-order eigenvalue perturbation $\pi^{(1)}$ depends on the perturbation matrices $\bDelta_{L}$, $\bDelta_{L_s}$ corresponding to the Loewner matrices. In the upcoming subsections, we will discuss two different cases.

	\subsection{Sensitivity of the Loewner pencil eigenvalues with unstructured perturbations}

	In this subsection we consider the case of unstructured perturbation. The strength of the perturbation can be quantified by the norm of the perturbation matrix. We will show that in this case, i.e., unstructured perturbation, the sensitivity depends on the numerical condition of the Loewner pencil. Assume that the norm of perturbation matrices is bounded as follows:
	\begin{align}
	\Vert \bDelta_L \Vert_2 \le \epsilon \omega_0 ,\quad
	\Vert \bDelta_{L_s} \Vert_2 \le \epsilon\omega_1  .
	\end{align}
	where  $\omega_i\geq0$, $i=0,1$ are the weights. Straightforward calculations yield to the bound
	\begin{align}
	\vert \pi^{(1)} \vert &=  \left\vert \frac{\bp^T\left( \bDelta_{L_s} - \pi \bDelta_L \right) \bq}{\bp^T \IL \bq} \right\vert \leq   \frac{\|\bp\|_2\left| \bDelta_{L_s} - \pi \bDelta_L \right| \|\bq\|}{|\bp^T \bE \bq|}  \le \epsilon \rho,
	\end{align}
	where the sensitivity $\rho$ is given by
	\begin{equation}\label{sensitivity}
	\rho=(\omega_0+|\pi|\omega_1)\,\frac{\|\bp\|_2\|\bq\|_2}{|\bp^T\IL\bq|}.
	\end{equation}
	This formula is well known in the literature and can be found, e.g., in \cite{Luis19}.
	
	To elaborate on this expression we will make use of the system realization quoted in Lemma \ref{lemma:2}. As already mentioned, the right and left eigenvectors of the Loewner matrix pencil given by $(\sIL, \IL)$ corresponding to the eigenvalue $\pi_i$ are given, respectively, by
	$$ 
	\bq_i = (\mathcal{C}_{\lambda, \pi}^T)^{+} \mathbf{e}_i, \quad \text{and} \quad
	\bp_i = (\mathcal{C}_{\mu, \pi}^T)^{+} \mathbf{e}_i.
	$$
	Consequently, it follows that
	$\bp_i^T \IL \bq_i$ $=$ $-\mathbf{e}_i^T (\mathcal{C}_{\mu, \pi})^{+}  \cC_{\mu,\pi} \bGamma \cC^T_{\lambda,\pi} (\mathcal{C}_{\lambda, \pi}^T)^{+} \mathbf{e}_i$ 
	$=$ $-\gamma_i$.
	Next, we make a particular choice of weights given by $\omega_0=\Vert \IL \Vert_2$
	and $\omega_1=\Vert \sIL \Vert_2$ (this is typical for such problems).
	Hence, by applying the general formula in (\ref{sensitivity}), the sensitivity of the eigenvalue $\pi_i$ is given as follows:
	\begin{equation} \label{equ:sen-uncon}
	\rho_i = \frac{1}{|\gamma_i|}  \Vert \bp_i \Vert_2  ( \vert \pi_i \vert \Vert  \IL \Vert_2 +  \Vert \IL_s \Vert_2)  \Vert \bq_i \Vert_2.
	\end{equation}
	By using the factorizations in (\ref{equ:facsiso}), we get that $\IL = - \cC_{\mu,\pi} \bGamma \cC^T_{\lambda,\pi}$ and $\IL_s = -\cC_{\mu,\pi} \bPi \bGamma \cC^T_{\lambda,\pi}$, and hence, the formula in (\ref{equ:sen-uncon}) is equivalently rewritten as 
	\begin{equation} \label{equ:sen-unco2}
	\rho_i = \frac{1}{|\gamma_i|}  \Vert (\mathcal{C}_{\mu, \pi}^T)^{+} \mathbf{e}_i \Vert_2  \Big( \vert \pi_i \vert \Vert \cC_{\mu,\pi} \bGamma \cC^T_{\lambda,\pi} \Vert_2 +  \Vert\cC_{\mu,\pi} \bPi \bGamma \cC^T_{\lambda,\pi} \Vert_2 \Big) \Vert (\mathcal{C}_{\lambda, \pi}^T)^{+} \mathbf{e}_i \Vert_2.
	\end{equation}
	
	\subsubsection{Deriving error bounds} 
	
	By using the same previously mentioned factorizations of the Loewner matrices, it follows that the following inequalities hold:
	\begin{align}\label{bounds_Loew_Loews}
	\begin{split}
	\IL &= - \cC_{\mu,\pi} \bGamma \cC^T_{\lambda,\pi} \Rightarrow \| \IL \|_2  
	\leq   \|  \cC_{\mu,\pi} \|_2 \Vert \bGamma \Vert_2 \| \cC^T_{\lambda,\pi}\|_2,
	\\
	\IL_s &= -\cC_{\mu,\pi} \bPi \bGamma \cC^T_{\lambda,\pi} \Rightarrow 
	\| \IL_s \|_2 \leq   \|  \cC_{\mu,\pi} \|_2 \Vert \bPi \bGamma \Vert_2  \| 
	\cC^T_{\lambda,\pi} \|_2. 
		\end{split}
	\end{align}
	By plugging in the inequalities from (\ref{bounds_Loew_Loews}) into the extended formula (\ref{equ:sen-unco2}), we get that
	\begin{align}\label{equ:bound_prelim}
	\rho_i  &\le  \frac{1}{|\gamma_i|}  \Vert \bp_i \Vert_2 \Vert  \cC_{\mu,\pi} \Vert_2  ( \vert \pi_i \vert \Vert \bGamma \Vert_2   +   \Vert \bPi \bGamma \Vert_2  ) \Vert \cC^T_{\lambda,\pi} \Vert_2 \Vert \bq_i \Vert_2 \\
	&\le \frac{\vert \pi_i \vert \Vert \bGamma \Vert_2   +   \Vert \bPi \bGamma \Vert_2 }{|\gamma_i|}  \Vert \bp_i \Vert_2 \Vert  \cC_{\mu,\pi} \Vert_2    \Vert \cC^T_{\lambda,\pi} \Vert_2 \Vert \bq_i \Vert_2 \nonumber \\
	&\Rightarrow \rho_i \le \frac{\vert \pi_i \vert \Vert \bGamma \Vert_2   +   \Vert \bPi \bGamma \Vert_2 }{|\gamma_i|}  \Vert(\mathcal{C}_{\mu, \pi}^T)^{+} \mathbf{e}_i \Vert_2 \Vert  \cC_{\mu,\pi} \Vert_2    \Vert \cC^T_{\lambda,\pi} \Vert_2 \Vert (\mathcal{C}_{\lambda, \pi}^T)^{+} \mathbf{e}_i \Vert_2.
	\end{align}
	Using the fact that $\Vert \bX \bfe_i \Vert_2 \le \Vert \bX  \Vert_2$, 
	the bound in (\ref{equ:bound_prelim}) is rewritten as follows:
\begin{align*}
	\rho_i  &\le \frac{\vert \pi_i \vert \Vert \bGamma \Vert_2   +   
		\Vert \bPi \bGamma \Vert_2 }{\|\gamma_i|}  \Vert(\mathcal{C}_{\mu, \pi}^T)^{+}  
	\Vert_2 \Vert  \cC_{\mu,\pi} \Vert_2    \Vert \cC^T_{\lambda,\pi} \Vert_2 
	\Vert (\mathcal{C}_{\lambda, \pi}^T)^{+}  \Vert_2 \\
	&\Rightarrow \rho_i \le \frac{\vert \pi_i \vert \Vert \bGamma \Vert_2   +   
		\Vert \bPi \bGamma \Vert_2 }{|\gamma_i|}  \kappa(\cC_{\mu,\pi})   
	\kappa(\cC^T_{\lambda,\pi}),
	\end{align*}
	where $\kappa(\bX)$ denotes the condition matrix of matrix $\bX$. Next, by using the identities
	$$
	\Vert  \bGamma \Vert_2 = \max_j(\vert \gamma_j \vert), \ \ 
	\Vert \bPi \bGamma \Vert_2 = \max_j(\vert \pi_j \gamma_j \vert),
	$$
	and by substituting these equalities into the inequality above, the upper bound can be further rewritten as:
	\begin{equation}\label{equ:bound1}	
	\rho_i   \le \zeta_i \, \kappa(\cC_{\mu,\pi})
	\,\kappa(\cC_{\lambda,\pi}),
	\end{equation}
	where $\zeta_i = \displaystyle \frac{1}{\vert \gamma_i \vert} \left(\vert \pi_i \vert \displaystyle \max_j(\vert \gamma_j \vert)  
	+ \displaystyle \max_j(\vert \pi_j \gamma_j \vert) \right),
	$
	for all $1 \le i \le n$. \\
	
	In what follows we will derive bounds for the the vector of unstructured sensitivities $\boldsymbol{\rho} \in \IR^n$, defined as $ \boldsymbol{\rho} = \left[ \begin{matrix}
	\rho_1 & \rho_2 & \cdots & \rho_n
	\end{matrix} \right]^T$. Similarly, let $\boldsymbol{\zeta} \in \IR^n$. Then, from (\ref{equ:bound1}), it readily follows that
	\begin{equation}\label{equ:bound2}
	\Vert \boldsymbol{\rho} \Vert_2 \le \Vert \boldsymbol{\zeta} \Vert_2 \, \kappa(\cC_{\mu,\pi})
	\,\kappa(\cC_{\lambda,\pi}). 
	\end{equation}
	Next, by making use of the inequality $x y \le \frac{1}{2} (x^2+y^2)$, one can rewrite (\ref{equ:bound_prelim}) as
	\begin{align}\label{equ:bound_prelim2}
	\rho_i \le \frac{1}{2} \zeta_i \Big{(} \Vert(\mathcal{C}_{\mu, \pi}^T)^{+} \mathbf{e}_i \Vert_2^2 \Vert  \cC_{\mu,\pi} \Vert_2^2 +    \Vert \cC^T_{\lambda,\pi} \Vert_2^2 \Vert (\mathcal{C}_{\lambda, \pi}^T)^{+} \mathbf{e}_i \Vert_2^2 \Big{)}
	\end{align}
	Denote with $\zeta^{(\textrm max)} = \displaystyle \max_i(\zeta_i)$. By summing up the inequalities in (\ref{equ:bound_prelim2}), one can write that
	\begin{align}
	\sum_{i=1}^q \rho_i \le \frac{1}{2} \zeta^{(\textrm max)} \Big{(} \Vert  \cC_{\mu,\pi} \Vert_2^2 \sum_{i=1}^q \Vert(\mathcal{C}_{\mu, \pi}^T)^{+} \mathbf{e}_i \Vert_2^2  +    \Vert \cC^T_{\lambda,\pi} \Vert_2^2 \sum_{i=1}^q \Vert (\mathcal{C}_{\lambda, \pi}^T)^{+} \mathbf{e}_i \Vert_2^2 \Big{)},
	\end{align}
	and by using that $\sum_{i=1}^q \Vert \bX \bfe_i \Vert_2^2 = \Vert \bX \Vert_{\textrm F}^2 $, it follows that
	\begin{align}\label{equ:bound3}
	\displaystyle \Vert \boldsymbol{\rho} \Vert_1 \le \frac{1}{2} \zeta^{(\textrm max)} \Big{(} \Vert  \cC_{\mu,\pi} \Vert_2^2  \Vert(\mathcal{C}_{\mu, \pi}^T)^{+} \Vert_{\textrm F}^2  +    \Vert \cC^T_{\lambda,\pi} \Vert_2^2 \Vert (\mathcal{C}_{\lambda, \pi}^T)^{+}  \Vert_{\textrm F}^2 \Big{)}.
	\end{align}

	\subsubsection{Connections with pseudospectra}
	
	Pseudospectra represent important tools for the numerical analysis of uncertain linear systems, eigenvalue perturbations or stability study.
	
	\begin{definition}
		Given a matrix $\bA \in \mathbb{C}^{n \times n}$ and a positive real constant $\epsilon >0$, the $\epsilon$-pseudospectrum of $\bA$ is:
		\begin{equation}\label{equ:pseudoA_def}
		\sigma_{\epsilon}(\bA) = \{ z \in \mathbb{C} \vert \text{is an eigenvalue of} \ \bA + \bGamma \ \text{for some} \ \bGamma \in \mathbb{C}^{n \times n} \ \text{with} \ \Vert \bGamma \Vert < \epsilon \}.
		\end{equation}
	\end{definition}
	Note that, for all $\epsilon > 0$, $\sigma_\epsilon(\bA)$ is a bounded, open subset of the complex plane that contains the eigenvalues of $\bA$.
	
	The concept of pseudospectrum has been extended over the years to cope with more general eigenvalue problems and dynamical systems. More precisely, we are interested in extensions of (\ref{equ:pseudoA_def}) to matrix pencils $(\bA,\bE)$. In \cite{Embree19}, one definition for the pseudospectrum of matrix pencil is mentioned. In order to allow matrix perturbations of both $\bA$ and $\bE$ to be scaled independently, this definition below includes two additional parameters, denoted with $\nu$ and $\delta$.
	\begin{definition}
		Let $\nu, \delta > 0$. For matrix pencil $\left(\bA, \bE\right) \in \mathbb{C}^{n \times n}$ and $\forall \epsilon \ge 0$, the $\epsilon$-$\left(\nu, \delta\right)$-pseudospectrum $\sigma_\epsilon^{\left(\nu, \delta\right)} \left(\bA, \bE\right)$ of the matrix pencil $\pi \bE- \bA$ is the set
		\begin{align}\label{equ:pseudoAE_def}
		\begin{split}
		\sigma_\epsilon^{\left(\nu, \delta\right)} \left(\bA, \bE\right) =& 
		\{
		\pi \in \mathbb{C} ~ is ~ an ~eigenvalue~of~the~pencil~\pi \left(\bE + \epsilon \bDelta_{\bE}\right) - \left( \bA + \epsilon\bDelta_{\bA}\right) \\
		& for~some~\bDelta_{\bE}, \bDelta_{\bA} \in \mathbb{C}^{n\times n} ~ with~\Vert \bDelta_{\bE} \Vert = \nu, \Vert\bDelta_{\bA} \Vert = \delta \}.
		\end{split}
		\end{align}
	\end{definition}
	As also illustrated in \cite{Embree19}, the pseudospectra provides a useful tool to explore the sensitivity of eigenvalues of matrix pencils. It is easy to see that the special points of the pseudospectra for $\epsilon = 0$ are precisely the eigenvalues of the matrix pencil $\left(\bA, \bE\right)$.
	Moreover, the slope of the pseudospectra around an eigenvalue $\pi$ can be used as a scale for eigenvalue sensitivity. If the slope is large, it means that the eigenvalue perturbation $\delta \pi$ is small for large perturbations applied to the pencil. Similarly, when the slope is small, it means that the eigenvalue perturbation $\delta \pi$ is large with a small perturbation of the matrix pencil.
	The sensitivity value $\rho_i$ introduced in (\ref{equ:sen-uncon}) is connected to the pseudospectra in the sense that both can be used to measure eigenvalue perturbations when the matrix pencil is perturbed and quantified by its norm. 
	
	For $\epsilon_2>\epsilon_1 > 0$, it follows that the slope of the pseudospectra near the eigenvalue $\pi$ can be defined as
	\begin{equation}\label{equ:slope_def}
	\xi = \lim_{\epsilon_1, \epsilon_2 \rightarrow 0} \left\vert \frac{\pi_2^{(1)} - \pi_1^{(1)}}{\epsilon_2 - \epsilon_1} \right\vert,
	\end{equation}
	where $\pi_i$ is the eigenvalue of matrix pencil $\left(\bE + \epsilon_i \bDelta_{\bE},  \bA + \epsilon_i\bDelta_{\bA}\right)$, for $i=1,2$ that is closest to $\pi$. Using the first order approximation of the eigenvalue perturbation, we obtain that
	\begin{align}\label{equ:bound_pi}
	\pi_i^{(1)} &=  \left\vert \frac{\bp^T\left( \epsilon_{i}\bDelta_\bA - \pi \epsilon_{i}\bDelta_\bE \right) \bq}{\bp^T \bE \bq} \right\vert =  \epsilon_{i} \left\vert \frac{\bp^T\left( \bDelta_\bA - \pi \bDelta_\bE \right) \bq}{\bp^T \bE \bq} \right\vert \le \epsilon_i \rho, \ \forall i=1,2.
	\end{align}
	Then, from (\ref{equ:slope_def}) and (\ref{equ:bound_pi}), it automatically follows that $\xi \le \rho$. Hence, the sensitivity $\rho$ provides an upper bound for the slope of the pseudospectra in the neighborhood of eigenvalue $\pi$. 
	
	It is also mentioned in \cite{Embree19} that analyzing structured pseudospectra could provide additional insight, i.e. the perturbed matrices have the same structure as the unperturbed (they are also Loewner matrices). Additionally, it could also be  helpful to explore the eigenvalue sensitivity of Loewner pencils in the case of noisy data. 
	The efficient computation for pseudospectra is still a matter of ongoing research. A recent result is provided in \cite{Embree19}, e.g., in Section 4, where the authors propose methods to accelerate this computation for (large) structured Loewner pencils and hence reduce the cost from $O(n^3)$ to $O(n^2)$ operations. In the next section, we explore the eigenvalue sensitivity of the Loewner pencil with respect to a structured perturbation.

	\subsection{Sensitivity of the eigenvalues with structured perturbations}
	
	In this section, we will study structured perturbations, i.e., the perturbation matrices have a particular structure. When the data are perturbed, e.g., the measurements are corrupted by additive noise, the perturbation matrices are indeed Loewner matrices. Recently, the robustness of the Loewner framework with respect to noise in the transfer-function values was studied in \cite{Zlatko19}. There, a statistical analysis was provided for bounding the deviation between the transfer function of noisy Loewner models and that of Loewner models without noise (see, e.g., Theorem 1 in Section 4). The eigenvalue sensitivity analysis with respect to noise in the data is significant for assessing the robustness of the Loewner models. In what follows, we will treat only the case of SISO systems (the MIMO case follows equivalently by adapting the factorization formulas as was shown in Section \ref{sec:MIMO_str_prop}).
	
	\subsubsection{Distinct left and right interpolation points  ($\bmu \ne \blambda$)}
	\label{sec:Sens_struct_distinct}
	
	Introduce the perturbed transfer function defined as  $\bar{\bH}(s) = \bH(s) (1 + \epsilon_s)$, and consider the left and right measurements corresponding to this transfer function under the influence of perturbation, i.e., 
	\begin{align}	
	\bar{\bH}(\epsilon_{\mu_i}) = \bH(\epsilon_{\mu_i}) (1 + \epsilon_{\mu_i}) = \bv_i (1 + \epsilon_{\mu_i}), \ \ \text{and}, \ \ \bar{\bH}(\epsilon_{\lambda_j}) = \bH(\epsilon_{\lambda_j}) (1 + \epsilon_{\lambda_j})= \bw_j (1 + \epsilon_{\lambda_j}),
	\end{align}
	for $1 \le i \le q$ and $1 \le j \le k$. Define the diagonal matrices $\cV \in \mathbb{C}^{q \times q}$ and $\cW \in \mathbb{C}^{k \times k}$ as
	$$
	\mathcal{V} = \mbox{diag}(\bv_1,\bv_2,\cdots,\bv_q), \ \ \mathcal{W} = \mbox{diag}(\bw_1,\bw_2,\cdots,\bw_k).
	$$
	Then, the following relations hold
	\begin{equation}
	\IL = \mathcal{V} \mathcal{C}_{\mu, \lambda} - \mathcal{C}_{\mu, \lambda}
	\mathcal{W},~~ \sIL = \bM \mathcal{V} \mathcal{C}_{\mu, \lambda} - \mathcal{C}_{\mu, \lambda} \mathcal{W} \bLambda.
	\end{equation}
	
	Similarly, the perturbation matrices can be written as
	\begin{align}\label{equ:per_mat_def}
	\begin{split}
	\bDelta_{L} &= \mbox{diag}(\epsilon_{\mu_1}, \epsilon_{\mu_2}, \cdots, \epsilon_{\mu_q}) 
	\mathcal{V} \mathcal{C}_{\mu, \lambda} - \mathcal{C}_{\mu, \lambda} \mathcal{W} 
	\mbox{diag}(\epsilon_{\lambda_1}, \epsilon_{\lambda_2}, 
	\cdots, \epsilon_{\lambda_k}),\\[1mm] 
	\bDelta_{L_s}  &= \mbox{diag}(\epsilon_{\mu_1}, \epsilon_{\mu_2}, \cdots, \epsilon_{\mu_q})
	\bM \mathcal{V} \mathcal{C}_{\mu, \lambda} - \mathcal{C}_{\mu, \lambda} \mathcal{W}
	\bLambda\mbox{diag}(\epsilon_{\lambda_1}, \epsilon_{\lambda_2}, \cdots, 
	\epsilon_{\lambda_k}).
	\end{split}
	\end{align}
	From (\ref{equ:per_mat_def}), it follows that the perturbation pencil is expressed as
	\begin{align}\label{equ:per_penc_def}
	s \bDelta_{L}  - \bDelta_{L_s}  & = \mbox{diag}(\epsilon_{\mu_1}, \epsilon_{\mu_2}, \cdots, \epsilon_{\mu_q}) (s\bI - \bM) \mathcal{V} \mathcal{C}_{\mu, \lambda} - \mathcal{C}_{\mu, \lambda} \mathcal{W} (s\bI - \bLambda)
	\mbox{diag}(\epsilon_{\lambda_1},\epsilon_{\lambda_2}, \cdots, \epsilon_{\lambda_k}).
	\end{align}
	By substituting (\ref{equ:per_penc_def}) into (\ref{equ:sensitivity}), it follows that the first order approximation of the eigenvalue perturbation corresponding to the pole $\pi_i$, for $1 \leq i \leq n$, is given by
	\begin{equation}	\label{equ:loewpert}
	\left.\begin{split}
	\pi_i^{(1)}~ &=\frac{ \bp_i^T\left(  \bDelta_{L_s}  -\pi_i \bDelta_{L}   \right) \bq_i}{\bp_i^T \IL \bq_i}=\frac{1}{\gamma_i} \bp_i^T\left(  \pi_i \bDelta_{L}  - \bDelta_{L_s}  \right) \bq_i \\
	&= \frac{1}{\gamma_i}\boldsymbol{\epsilon}_\mu^T \mbox{diag}(\bp_i) ( \pi_i \bI - \bM) \mathcal{V} \mathcal{C}_{\mu, \lambda} \bq_i - \frac{1}{\gamma_i}\bp_i^T\mathcal{C}_{\mu, \lambda} \mathcal{W} (\pi_i \bI - \bLambda)
	\mbox{diag}(\bq_i) \boldsymbol{\epsilon}_\lambda \\
	& = 
	\left[ \begin{matrix}
	\boldsymbol{\epsilon}_\mu^T &
	\boldsymbol{\epsilon}_\lambda^T
	\end{matrix}\right] \frac{1}{\gamma_i}\left[\begin{matrix}
	\mbox{diag}(\bp_i) ( \pi_i \bI - \bM) \mathcal{V} \mathcal{C}_{\mu, \lambda} \bq_i &   \mbox{diag}(\bq_i) (\pi_i\bI - \bLambda) \mathcal{W} \mathcal{C}_{\lambda,\mu}  \bp_i
	\end{matrix}\right]  \\
	& = \boldsymbol{\epsilon}^T \left(\cS \bfe_i \right) = \boldsymbol{\epsilon}^T \cS_i,\\
	\end{split}~\right.,
	\end{equation}
	where $\boldsymbol{\epsilon}_\mu = \left[\epsilon_{\mu_1}, \cdots, \epsilon_{\mu_q} \right]^T \in \IC^q$, ~$\boldsymbol{\epsilon}_\lambda = \left[\epsilon_{\lambda_1}, \cdots, \epsilon_{\lambda_k} \right]^T \in \IC^k$,  $\boldsymbol{\epsilon} = \left[ \begin{matrix}
	\boldsymbol{\epsilon}_\mu \\
	\boldsymbol{\epsilon}_\lambda
	\end{matrix}\right] \in \IC^{q+k}$, and $\cS = \left[ \begin{matrix}
	\cS_\mu \\
	\cS_\lambda
	\end{matrix}\right] \in \IC^{(q+k) \times n}$. Additionally, the $i$th column of matrix $\cS$ is denoted with $\cS_i$, while the $i$th columns of matrices $\cS_\mu \in \IC^{q \times n}$ and $\cS_\lambda \in \IC^{k \times n}$, for all $1 \leq i \leq n$, are given by: 
	\begin{align}
	\begin{cases}{\cS_\mu}_i =  \cS_\mu \bfe_i  = \frac{1}{\gamma_i}
	\begin{matrix}
	\mbox{diag}(\bp_i) (\pi_i\bI - \bM)\mathcal{V}\mathcal{C}_{\mu, \lambda}\bq_i   
	\end{matrix} \in \IC^{q}, \\[2mm] {\cS_\lambda}_i =  \cS_\lambda \bfe_i  = \frac{1}{\gamma_i}
	\begin{matrix}
	\mbox{diag}(\bq_i) (\pi_i\bI - \bLambda) \mathcal{W} \mathcal{C}_{\lambda,\mu}  \bp_i
	\end{matrix} \in \IC^{k}
	\end{cases}, \text{and} \ \ \cS_i = \left[ \begin{matrix}
	{\cS_\mu}_i \\
	{\cS_\lambda}_i
	\end{matrix}\right] \in \IC^{(q+k)}.
	\end{align}
	
	As shown in (\ref{equ:loewpert}), the eigenvalue perturbation is a linear combination of the left and right measurement noise.
	By considering noise in the left measurements, the effect on perturbing the poles can be hence quantified by the entries of matrix  $\cS_\mu \in \IC^{q \times n}$. Additionally, by choosing noisy right measurements, the perturbation of the poles is hence quantified by the entries of matrix  $\cS_\lambda \in \IC^{k \times n}$.
	
	
	Next, a simplified case is considered in order to better understand the structure of matrix $\cS$. We consider the case for which we quantify how perturbing the $j$th left measurement is affecting the $i$th pole $\pi_i$. Using that $\bfe_j^T \mbox{diag}(\bp_i) = \bp_i^T \bfe_j\bfe_j^T$, it follows that the $(j,i)$ entry of matrix  $\cS_\mu$ is explicitly given by the following formula
	\begin{align}\label{equ:Smu}
	\begin{split}
	\begin{array}{rcl}
	\left( \cS_\mu \right)_{j,i} = \bfe_j^T \cS_\mu \bfe_i
	&=& \bfe_j^T \frac{1}{\gamma_i}\mbox{diag}(\bp_i) (\pi_i\bI - \bM)\mathcal{V}\mathcal{C}_{\mu, \lambda}\bq_i=\frac{1}{\gamma_i}
	\bp_i^T\left[\bfe_j\bfe_j^T\,(\pi_i\bI-\bM)
	\mathcal{V} \cC_{\mu,\lambda}\right] \,\bq_i\\[2mm]
	&=&\frac{1}{\gamma_i}
	\bfe_i^T(\cC_{\mu,\pi})^+
	\left[\bfe_j\bfe_j^T\,(\pi_i\bI-\bM)
	\mathcal{V} \cC_{\mu,\lambda}\right] \,(\cC_{\lambda,\pi}^T)^+\bfe_i\\[2mm]
	&=&\frac{\left(\pi_i-\mu_j\right)\bv_j}{\gamma_i}
	\left(\bfe_i^T(\cC_{\mu,\pi})^+\bfe_j\right)
	\left(\bfe_j^T(\cC_{\mu,\lambda})(\cC_{\lambda,\pi}^T)^+\bfe_i\right).
	\end{array}
	\end{split}
	\end{align}
	Similarly, assume now that the $j^{th}$ right measurement is perturbed and we would like to measure the influence on the $i$th pole. We have that the $(j,i)$ entry of matrix  $\cS_\lambda$ is explicitly given as
	\begin{align}\label{equ:Sla}
	\begin{split}
	\begin{array}{rcl}
	\left( \cS_\lambda \right)_{j,i} = \bfe_j^T \cS_\lambda \bfe_i
	&=& \frac{\left(\pi_i-\lambda_j\right)\bw_j}{\gamma_i}
	\left(\bfe_i^T(\cC_{\lambda,\pi})^+\bfe_j\right)
	\left(\bfe_j^T(\cC_{\lambda,\mu})(\cC_{\mu,\pi}^T)^+\bfe_i\right).
	\end{array}
	\end{split}
	\end{align}
	In what follows, consider a special case for which the number of poles equals to the number of left and right interpolation points, i.e., $k=q=n$. The inverse of a square Cauchy matrix $\cC_{x,y} \in \mathbb{C}^{n \times n}$ is explicitly expressed as
	\begin{equation} \label{equ:inv_cauchy_def}
	\left(\cC_{x,y}^{-1}\right)_{i, j} = \frac{\prod_{k=1}^{n} \left(x_j - y_k\right) \prod_{k=1}^{n} \left(x_k - y_i\right)}{\left(x_j - y_i\right) \prod_{k \ne j} \left(x_j - x_k\right) \prod_{k \ne i} \left(y_k - y_i\right)}.
	\end{equation} 
	
	By substituting the result (\ref{equ:inv_cauchy_def}) in the formula (\ref{equ:Smu}), we obtain that the $(j,i)$ entry of matrix $\cS_\mu$ can be explicitly written in terms of the poles and the left and right interpolation points, as follows
	\begin{align}\label{equ:Smu_exp}
	\begin{split}
	\left( \cS_\mu \right)_{j,i}
	&=\frac{\left(\pi_i-\mu_j\right)\bv_j}{\gamma_i}
	\left(\bfe_\ell^T(\cC_{\mu,\pi})^+\bfe_i\right)
	\left(\bfe_i^T(\cC_{\mu,\lambda})(\cC_{\lambda,\pi}^T)^+\bfe_\ell\right) \\
	& = -\frac{ \bv_j}{\gamma_i} \frac{\prod_{k=1}^{n} \left(\mu_j - \pi_k\right)  \left(\pi_i - \mu_k\right)  \left( \pi_i - \lambda_k \right)}{\prod_{k \ne j} \left(\mu_j - \mu_k\right) \prod_{k \ne i} \left( \pi_i - \pi_k\right)^2}
	\left(
	\sum_{m=1}^{n} 
	\frac{\prod_{k\ne i}^{} \left(\pi_k - \lambda_m \right)}{\left(\mu_j - \lambda_m \right)
		\prod_{k \ne m} \left(\lambda_k - \lambda_m\right)} \right).
	\end{split}
	\end{align}
	Similar derivations can be obtained for the $(j,i)$ entry of the matrix $\cS_\lambda$. We will illustrate the results presented in  formula (\ref{equ:Smu_exp}) by means of a couple of simplified scenarios. 
	\begin{example}
		Choose $n = q = k =1$ and $\bH(s) = \frac{\gamma_1}{s- \pi_1}$. Hence, one can then write
		\begin{equation}
		\left( \cS_\mu \right)_{1,1} = \frac{\bv_1}{\gamma_1} \frac{\left(\pi_{1}-\lambda_{1}\right)}{\left(\mu_{1}-\lambda_{1}\right)} \left(\mu_{1}-\pi_1\right)^2.
		\end{equation}
		Next, choose $n = q = k =2$ and $\bH(s) = \frac{\gamma_1}{s- \pi_1}+  \frac{\gamma_2}{s- \pi_2}$. Hence, the following hold for $1 \leq i \leq 2$
		\begin{align}
		\left( \cS_\mu \right)_{1,i} &=  \frac{\bv_i}{\gamma_i}   \frac{\left(\pi_{i}-\lambda_{1}\right)\,\left(\pi_{i}-\lambda_{2}\right)\,\left(\pi_{i}-\mu_{2}\right)}{\left(\mu_{1}-\lambda_{1}\right)\,\left(\mu_{1}-\lambda_{2}\right)\,\left(\mu_{1}-\mu_{2}\right)} \frac{\left(\mu_{1}-\pi_{2}\right)^2\,\left(\mu_{1}-\pi_{1}\right)^2}{\left(\pi_{1}-\pi_{2}\right)^2}, \\
		\left( \cS_\mu \right)_{2,i} &=  \frac{\bv_i}{\gamma_i}   \frac{\left(\pi_{i}-\lambda_{1}\right)\,\left(\pi_{i}-\lambda_{2}\right)\,\left(\pi_{i}-\mu_{1}\right)}{\left(\mu_{2}-\lambda_{1}\right)\,\left(\mu_{2}-\lambda_{2}\right)\,\left(\mu_{2}-\mu_{1}\right)} \frac{\left(\mu_{2}-\pi_{2}\right)^2\,\left(\mu_{2}-\pi_{1}\right)^2}{\left(\pi_{1}-\pi_{2}\right)^2}.
		\end{align}
		Similar formulas can be derived for $\left( \cS_\lambda \right)_{j,i}$, but will be omitted here.
	\end{example}

	\begin{definition}
		Let $\eta_{(j, i)}$ be the structured sensitivity  defined for eigenvalue $\pi_i$ with respect to perturbing the $j^{th}$ left or right measurement. The value $\eta_{(j, i)}$ is explicitly given by
		\begin{equation} \label{equ:sen-con}
		\eta_{(j,i)} = \vert \left(\cS \right)_{j,i} \vert =  \begin{cases}
		\vert \left(\cS_\mu \right)_{j,i} \vert, \ \ \text{if} \ \ 1\leq j \leq q,\\
		\vert \left(\cS_\lambda \right)_{j,i} \vert, \ \ \text{if} \ \  q+1\leq j \leq q+k,
		\end{cases}
		\end{equation}	
		Let $\cN \in \IR^{(q+k)\times n}$ be the matrix containing all structured sensitivity values in (\ref{equ:sen-con}). Furthermore, we split matrix $\cN = \left[ \begin{matrix}
		\cN_\mu \\
		\cN_\lambda
		\end{matrix}\right]$ into two sub-matrices, $\cN_{\mu} \in \IR^{q\times n}$ and  $\cN_{\lambda} \in \IR^{k\times n}$, corresponding to the left and right measurements, so that $\cN_{\mu} = \vert \cS_\mu \vert$, and $\cN_{\lambda} = \vert \cS_\lambda \vert$.
	\end{definition}
	
	\begin{remark}
		It is to be noted that the structured sensitivity formula given in (\ref{equ:sen-con}) is proportional to the left and right measurements. Alternatively, we could also introduce the absolute structured sensitivity  defined for eigenvalue $\pi_i$ with respect to perturbing the $j^{th}$ left or right measurement. This is denoted by $\overline{\eta}_{(j, i)}$, and is explicitly given by 
		\begin{equation}
		\overline{\eta}_{(j, i)} = \begin{cases}
		\eta_{(j, i)}/\bv_j, \ \ \text{if} \ \ 1\leq j \leq q,\\
		\eta_{(j, i)}/\bw_j, \ \ \text{if} \ \ q+1\leq j \leq q+k.
		\end{cases}
		\end{equation}
	\end{remark}
	
	The variance of the eigenvalue perturbation $\pi_i^{(1)}$ is given by
	$$\mbox{Var}(\pi_i^{(1)}) = {E} \left[ \left(\cS_i\right)^T\boldsymbol{\epsilon} \boldsymbol{\epsilon}^H\left(\cS_i\right)^* \right] = 
	\left(\cS_i\right)^T{E}  \left[ \boldsymbol{\epsilon} \boldsymbol{\epsilon}^H \right]\left(\cS_i\right)^*.
	$$
	
	Assuming that the noise is Gaussian (which is typical the case in many practical situations), it follows that  ~${E}  \left[ \boldsymbol{\epsilon} \boldsymbol{\epsilon}^H \right] = 
	\sigma_{\bepsilon}^2\mathbf{I}$, and hence, we can write the variance and the standard deviation of $\pi_i^{(1)}$ as
	\begin{equation}\label{equ:std-gaussian-loew}
	\mbox{Var}(\pi_i^{(1)}) = \sigma_{\bepsilon}^2 \left\Vert \cN_i \right\Vert_2^2, \ \ \sigma(\pi_i^{(1)}) = \sigma_{\bepsilon} \left\Vert \cN_i \right\Vert_2.
	\end{equation}
	Hence, conclude that $\left\Vert \cN_i \right\Vert_2$ determines the standard deviation of eigenvalue perturbation $\pi_i^{(1)}$.
	
	\begin{definition}
		The eigenvalue sensitivity $\bbeta_i$ is defined for the eigenvalue $\pi_i$  with respect to a structured perturbation, as the norm of column vector $\cN_i$. Hence, introduce the vector $\bbeta \in \IR^n$ such that its $i$th entry is
		\begin{equation} \label{equ:sen-con1}
		\bbeta_i = \left\Vert \cN_i \right\Vert_2 = \sqrt{\eta_{( 1,i)}^2 +\eta_{(2,i)}^2 + \ldots + \eta_{(q+k,i)}^2}.
		\end{equation}
	\end{definition}

	\begin{example}\label{ex:4.2}
		Consider the system characterized by the following realization:
		\begin{equation}\label{real:ex4.2}
		\bA =
		\left[\!\begin{array}{rrrrr} -1 & 0 & 0 & 0 & 0\\[1mm] 0 & -1 & 1 & 0 & 0\\[1mm] 0 & -1 & -1 & 0 & 0\\[1mm] 0 & 0 & 0 & -\frac{1}{2} & \frac{\sqrt{3}}{2}\\[1mm] 0 & 0 & 0 & -\frac{\sqrt{3}}{2} & -\frac{1}{2} \end{array}\!\right],~
		\bB =\left[\!\begin{array}{c} 1\\[1mm] 1\\[1mm] 0\\[1mm] 1\\[1mm] 0 \end{array}\!\right],~
		\bC =\left[\!\begin{array}{ccccc} 1 & 0 & 1 & 0 & \frac{2\,\sqrt{3}}{3} \end{array}\!\right].
		\end{equation}
		The poles, residues and the transfer function of this system are as follows 
		\small
		$$
		\begin{cases}
		\bpi~=~
		\left[\!\begin{array}{rrrrr} -1 & -1-\mathrm{i} & -1+\mathrm{i} & -\frac{1}{2}-\frac{\sqrt{3}}{2}\mathrm{i}& -\frac{1}{2}+\frac{\sqrt{3}}{2}\mathrm{i}\end{array}\!\right],\\\
		\bgamma=
		\left[\!\begin{array}{rrrrr} 1& -\frac{1}{2}{}\mathrm{i}& \frac{1}{2}{}\mathrm{i}& -\frac{\sqrt{3}}{3}\mathrm{i}& \frac{\sqrt{3}}{3} \mathrm{i}\end{array}\right], 	\end{cases} \hspace{-5mm}\Rightarrow~~
		\bH(s)= \frac{s^4+s^3-2\,s-1}{\left(s+1\right)\,\left(s^2+2\,s+2\right)\,\left(s^2+s+1\right)},
		$$
		\normalsize
		where $\bpi$ is the vector of poles and $\bgamma$ the vector of residues. Choose the right/left interpolation points as:
		$$
		\blambda =\left[\begin{array}{ccccc} \frac{2}{9} & \frac{4}{9} & \frac{6}{9} & \frac{8}{9} & \frac{10}{9} \end{array}\right],~\bmu = -\blambda^T.
		$$
		Then, put together the various Cauchy matrices:
		$\cC_{\mu,\pi}$,~
		$\cC_{\mu,\lambda}$,~ and
		$\cC_{\pi,\lambda}$.  Hence, compute the matrix $\cN_\mu \in \IR^{5 \times 5}$ of sensitivity associated to the left measurements as well as matrix $\cN_\lambda \in \IR^{5 \times 5}$ of sensitivity associated to the right measurements, as follows
		\footnotesize
		$$
		\frac{\cN_\mu}{10^2 } = \left[\begin{array}{rrrrr}
		0.0068  &  4.1910  &   4.1910   & 0.7128   &  0.7128 \\
		0.0016  &   0.7972  &   0.7972   &  0.1283   &  0.1283 \\
		0.0012   &  0.3757  &   0.3757   &  0.0548   &  0.0548 \\
		0.0011   &  0.1224  &   0.1224   &  0.0158  &   0.0158 \\
		0.0003   &  0.0375   &  0.0375   &  0.0043   &  0.0043
		\end{array} \right], \ \ \ \frac{\cN_\lambda}{10^3} =  
		\left[\begin{array}{rrrrr}
		0.0041  &  3.2017  &  3.2017   & 0.5474  &  0.5474 \\
		0.0093  &  7.6490  &  7.6490  &  1.2804  &  1.2804 \\
		0.0086  &  7.3720  &  7.3720  &  1.2041  &  1.2041 \\
		0.0031  &  2.7251  &  2.7251  & 0.4345 &   0.4345 \\
		0.0001  &  0.1169  &  0.1169  & 0.0182  &  0.0182
		\end{array} \right].
		$$
		\normalsize
		By examining the entries of $\cN_\mu$, it follows that the pair of poles $-1\pm i$ are the most sensitive, especially when perturbing the first left measurement value. This can be observed in the $(1,2)$ and $(1,3)$ entries of matrix $\cN_\mu$.
		
		Similarly, from the entries of $\cN_\lambda$, we conclude that the pair of poles $-1\pm i$ are the most sensitive. The largest entry is obtained when perturbing the second right measurement value. Again, this can be observed in the $(2,2)$ and $(2,3)$ entries of matrix $\cN_\lambda$.
		Additionally, compute the vector $\bbeta \in \IR^5$ as in (\ref{equ:sen-con1}):
		\begin{equation}
		\bbeta = 10^4 \left[ \begin{matrix}
		0.0014  &  1.1434  &  1.1434  &  0.1893  &  0.1893
		\end{matrix} \right]^T.
		\end{equation}
		From the above result, we indeed draw the conclusion that the pair of poles $-1\pm i$, is by far the most sensitive, also in an $\ell_2$ way (since $\bbeta_2$ and $\bbeta_3$ are the largest entries).
		
		Finally, it is to be noted that the sensitivities computed for matrix $\bA$ corresponding to the canonical realization in (\ref{real:ex4.2}) are all ones. In this case, the pencil $(\bA,\bE)$ simplifies to matrix $\bA$, since $\bE$ is the identity matrix. Hence, the choice of interpolation points can greatly influences the sensitivity computations.
		
	\end{example}

	\subsubsection{Same left and right interpolation 
		points ($\bmu = \blambda$)}
	\label{sec:Sens_struct_same}
	
	For the case $\bmu = \blambda$, the diagonal entries of the Loewner matrix are actually derivatives of transfer function $\bH(s)$. Considering that the measurements are perturbed, we assume that $\bar{\bH}(s) = \bH(s) (1 + \epsilon_s)$, and, additionally that $\bar{\bH}'(s) = \bH'(s) (1 + \epsilon'_s)$. Here, $\bH'(s)$ denotes the derivative of $\bH(s)$ w.r.t s, i.e.,   $\bH'(s) = \frac{d}{ds} \bH(s)$. The Loewner and shifted Loewner matrices can be written as follows
	\begin{equation}
	\IL = \mathcal{V} \mathcal{C}_{\mu, \mu} - \mathcal{C}_{\mu, \mu}
	\mathcal{V} +  \mathcal{V}',~~ \sIL = \bM \mathcal{V} \mathcal{C}_{\mu, \mu} - \mathcal{C}_{\mu, \mu} \mathcal{V} \bM + \bM \mathcal{V}' + \mathcal{V},
	\end{equation}
	where
	$\small
	\mathcal{V}' = \mbox{diag}(\bv_1',\bv_2',\cdots,\bv_q')$ with $\bv_i' = \bH'(\mu_i)$. Additionally,  note that $\mathcal{C}_{\mu, \mu}(i,i) = 0$ for all $i =1,\ldots,q $. The perturbation matrices $\bDelta_{L}$ and $\bDelta_{L_s}$ can be written as
	\begin{align}\label{equ:per_mat_def2}
	\begin{split}
	\bDelta_{L} &= \mbox{diag}(\epsilon_{\mu_1}, \cdots, \epsilon_{\mu_q})
	\mathcal{V} \mathcal{C}_{\mu, \mu} - \mathcal{C}_{\mu, \mu} \mathcal{V} 
	\mbox{diag}(\epsilon_{\mu_1}, \cdots, \epsilon_{\mu_q}) +  \mathcal{V}'\mbox{diag}(\epsilon_{\mu_1}', \cdots, \epsilon_{\mu_q}'), 
	\\[1mm] 
	\bDelta_{L_s}&= \mbox{diag}(\epsilon_{\mu_1}, \cdots, \epsilon_{\mu_q})
	\bM \mathcal{V} \mathcal{C}_{\mu, \mu} - \mathcal{C}_{\mu, \mu} \mathcal{V}
	\bM\mbox{diag}(\epsilon_{\mu_1}, \cdots, \epsilon_{\mu_q}) +  \mathcal{V}'\bM \mbox{diag}(\epsilon_{\mu_1}', \cdots, \epsilon_{\mu_q}') \\
	&+ \mbox{diag}(\epsilon_{\mu_1}, \cdots, \epsilon_{\mu_q}) \cV. 
	\end{split}
	\end{align}
	Since $\bmu = \blambda$, the left/right eigenvectors of the Loewner pencil are the same, i.e., $\bp = \bq$. From (\ref{equ:per_mat_def2}), it follows that
	\begin{align*}\small
	s \bDelta_{L} - \bDelta_{L_s}  &= \mbox{diag}(\epsilon_{\mu_1}, \cdots, \epsilon_{\mu_q}) (s\bI - \bM) \mathcal{V} \mathcal{C}_{\mu, \mu} - \mathcal{C}_{\mu, \mu} \mathcal{V} (s\bI - \bM)
	\mbox{diag}(\epsilon_{\mu_1}, \cdots, \epsilon_{\mu_q}) \\
	+&   \mathcal{V}' (s\bI - \bM) \mbox{diag}(\epsilon_{\mu_1}', \cdots, \epsilon_{\mu_q}') - \mbox{diag}(\epsilon_{\mu_1}, \cdots, \epsilon_{\mu_q}) \cV.
	\end{align*}
	By substituting the above expression into (\ref{equ:sensitivity}), it follows that the first order approximation of the eigenvalue perturbation corresponding to $\pi_i$ is given by
	\begin{equation}
	\hspace*{-3mm}	
	\begin{split}
	\pi_i^{(1)} &= \frac{ \bp_i^T\left(  \bDelta_{L_s}  -\pi_i \bDelta_{L}   \right) \bq_i}{\bp_i^T \IL \bq_i}  = \frac{1}{\gamma_i} \bp_i^T\left(  \pi_i \bDelta_{L} - \bDelta_{L_s} \right) \bq_i \\
	&= 
	\frac{1}{\gamma_i} \left[
	\boldsymbol{\epsilon}_\mu^T \mbox{diag}(\bp_i) (\pi_i\bI - \bM) \mathcal{V} \mathcal{C}_{\mu, \mu}  \bq_i  - \bp_i^T\mathcal{C}_{\mu, \mu} \mathcal{V} (\pi_i\bI - \bM) \mbox{diag}(\bq_i) \boldsymbol{\epsilon}_\mu - \boldsymbol{\epsilon}_\mu^T \mbox{diag}(\bp_i) \cV \bq_i
	\right] 
	\\
	&+ \left[ 
	\bp_i^T \mathcal{V}' (\pi_i\bI-\bM) \mbox{diag}(\bq_i) \boldsymbol{\epsilon}_\mu' \right] \\
	&= \left[ \begin{matrix}
	\boldsymbol{\epsilon}_\mu^T &
	\boldsymbol{\epsilon}_\mu'^T
	\end{matrix}\right] \frac{1}{\gamma_i}\left[\begin{matrix}
	2\mbox{diag}(\bq_i) (\pi_i\bI - \bM)\mathcal{V}\mathcal{C}_{\mu, \mu}\bq_i - \mbox{diag}(\bq_i) \cV \bq_i & 
	\mbox{diag}(\bq_i)(\pi_i\bI - \bM) \mathcal{V}' \bq_i
	\end{matrix}\right] 
	\\ 
	&=  \boldsymbol{\epsilon}^T  \left(\cT \bfe_i \right) = 
	\boldsymbol{\epsilon}^T \cT_i,
	\end{split}
	\end{equation}\normalsize
	where $\boldsymbol{\epsilon}_\mu = \left[\epsilon_{\mu_1}, \cdots, \epsilon_{\mu_q} \right]^T$, ~$\boldsymbol{\epsilon}_\mu' = \left[\epsilon_{\mu_1}', \cdots, \epsilon_{\mu_q}' \right]^T,$
	$\small
	\boldsymbol{\epsilon} = \left[ \begin{matrix}
	\boldsymbol{\epsilon}_\mu \\
	\boldsymbol{\epsilon}_\mu'
	\end{matrix}\right] \in \IC^{2q}$ and $\cT = \left[ \begin{matrix}
	\cT_\mu \\
	\cT_\mu'
	\end{matrix}\right] \in \IC^{2q \times n}$. Additionally, the $i$th column of matrix $\cT$ is denoted with $\cT_i$, while the $i$th columns of matrices $\cT_\mu \in \IC^{q \times n}$ and $\cT_\mu' \in \IC^{q \times n}$, for all $1 \leq i \leq n$, are given by: 
	\begin{equation}
	\begin{cases}{\cT_\mu}_i =  \cT_\mu \bfe_i  = 
	\frac{1}{\gamma_i} 
	2\mbox{diag}(\bq_i) (\pi_i\bI - \bM)\mathcal{V}\mathcal{C}_{\mu, \mu}\bq_i - \frac{1}{\gamma_i} 
	\mbox{diag}(\bq_i) \cV \bq_i, \\[2mm]
	{\cT_\mu'}_i =  \cT_\mu' \bfe_i  = \frac{1}{\gamma_i} 
	\mbox{diag}(\bq_i)(\pi_i\bI - \bM) \mathcal{V}' \bq_i,
	\end{cases}  \text{and} \ \ \cT_i = \left[ \begin{matrix}
	{\cT_\mu}_i \\[2mm]
	{\cT_\mu'}_i
	\end{matrix}\right] \in \IC^{2q}.
	\end{equation}
	As shown above, the eigenvalue perturbation is a linear combination of noise on measurements of the transfer function and of the transfer function derivative. By considering noise in the measurements of the transfer function, the effect on perturbing the poles can be hence quantified by the entries of matrix  $\cT_\mu \in \IC^{q \times n}$. Additionally, by choosing perturbed derivative measurements, the perturbation of the poles is then quantified by the entries of matrix  $\cT_\mu' \in \IC^{q \times n}$.
	
	As before, a simplified is considered to gain better insight into the structure of matrix $\cT$. We analyze a specific case, i.e., by perturbing the $j$th measurement, we seek to quantify the influence on the $i$th pole. Using that $\bfe_j^T \mbox{diag}(\bq_i) = \bq_i^T \bfe_j\bfe_j^T$, it follows that the $(j,i)$ entry of matrix  $\cT_\mu$ is explicitly given by the following formula
	\begin{align}
	\left( \cT_\mu \right)_{j,i} &= \bfe_j^T {\cT_\mu}_i  
	= \bfe_j^T \frac{1}{\gamma_i}\mbox{diag}(\bq_i) \left[ 2(\pi_i\bI - \bM)\mathcal{V}\mathcal{C}_{\mu, \mu} - \cV \right]\bq_i=\frac{1}{\gamma_i}
	\bq_i^T\bfe_j\bfe_j^T \left[2(\pi_i\bI-\bM)
	\mathcal{V} \cC_{\mu, \mu} - \cV \right] \,\bq_i \nonumber \\[2mm]
	&=\frac{1}{\gamma_i}
	\bfe_i^T(\cC_{\mu,\pi})^+ \bfe_j\bfe_j^T
	\left[2(\pi_i\bI-\bM)
	\mathcal{V} \cC_{\mu, \mu} - \cV\right] \,(\cC_{\mu,\pi}^T)^+\bfe_i \nonumber \\[2mm]
	&=\frac{2\left(\pi_i-\mu_j\right)\bv_j}{\gamma_i}
	\left(\bfe_i^T(\cC_{\mu,\pi})^+\bfe_j\right)
	\left(\bfe_j^T(\cC_{\mu, \mu})(\cC_{\mu,\pi}^T)^+\bfe_i\right) - \frac{\bv_j}{\gamma_i} \left(
	\bfe_i^T(\cC_{\mu,\pi})^+ \bfe_j \right) \left( \bfe_j^T
	\,(\cC_{\mu,\pi}^T)^+\bfe_i \right).
		\end{align}
	Similarly, assume now that  the $j^{th}$ derivative measurement is perturbed. We have that the $(j,i)$ entry of matrix  $\cT_\mu'$ is explicitly given as
	\begin{equation}
	\begin{array}{rcl}
	\left( \cT_\mu' \right)_{j,i}
	&= \frac{\left(\pi_i-\mu_j\right)\bv'_j}{\gamma_i}
	\left(\bfe_i^T(\cC_{\mu,\pi})^+\bfe_j\right)^2.
	\end{array}
	\end{equation}
	
	\begin{definition} \label{def:sens_struct_same}
		We define $\eta_{(j, i)}$ as the structured sensitivity for the eigenvalue $\pi_i$ with respect to the perturbing the $j^{th}$ measurement, or on $j^{th}$  derivative measurement,  as follows
		\begin{equation} \label{equ:sen-con2}
		\eta_{(j,i)} = \vert \left(\cT \right)_{j,i} \vert =  \begin{cases}
		\vert \left(\cT_\mu \right)_{j,i} \vert, \ \ \text{if} \ \ 1\leq j \leq q,\\
		\vert \left(\cT_\mu' \right)_{j,i} \vert, \ \ \text{if} \ \  q+1\leq j \leq 2q.
		\end{cases}
		\end{equation}	
	\end{definition}
	All other definitions and formulas introduced in Section \ref{sec:Sens_struct_distinct} follow equivalently, for values $\eta_{(j, i)}$ as in Definition \ref{def:sens_struct_same}. For example, let $\cN_{\mu} \in \IR^{q\times n}$ and  $\cN_{\mu}' \in \IR^{q\times n}$, corresponding to the measurements and to the derivatives, respectively, so that $\cN_{\mu} = \vert \cT_\mu \vert$, and $\cN_{\mu}' = \vert \cT_\mu' \vert$.
	
	\begin{example}\label{ex:4.3}
		Consider the same test case as in Example \ref{ex:4.2}. The vectors of poles and of residues are given by:
		$$
		\begin{cases}
		\bpi=
		\left[\!\begin{array}{rrrrr} -1 & -1-\mathrm{i} & -1+\mathrm{i} & -\frac{1}{2}-\frac{\sqrt{3}}{2}\mathrm{i}& -\frac{1}{2}+\frac{\sqrt{3}}{2}\mathrm{i}\end{array}\!\right],\\
		\bgamma=
		\left[\!\begin{array}{rrrrr} 1& -\frac{1}{2}{}\mathrm{i}& \frac{1}{2}{}\mathrm{i}& -\frac{\sqrt{3}}{3}\mathrm{i}& \frac{\sqrt{3}}{3} \mathrm{i}\end{array}\right].	\end{cases}
		$$
		Choose identical right/left interpolation points as follows:
		$$
		\blambda =\left[\begin{array}{ccccc} \frac{2}{9} & \frac{4}{9} & \frac{6}{9} & \frac{8}{9} & \frac{10}{9} \end{array}\right]= \bmu^T.
		$$
		Then, compute the Cauchy matrices:
		$\cC_{\mu,\pi}$,~
		$\cC_{\mu,\mu}$,~ and
		$\cC_{\pi,\mu}$ and the matrix $\cN_\mu \in \IR^{5 \times 5}$ of sensitivity associated to the measurements, as well as matrix $\cN_\mu' \in \IR^{5 \times 5}$ of sensitivity associated to the derivatives, as follows
		\footnotesize
		$$
		\frac{\cN_\mu}{10^8} =  \left[\begin{array}{rrrrr}
		0.0434  & 0.0897 &  0.0897  &  0.0079  &  0.0079 \\
		0.7100  & 1.5651 &  1.5651  & 0.1340   & 0.1340 \\
		0.9711 &  2.2717 &  2.2717 &  0.1872  &  0.1872 \\
		0.3703 &  0.8327 &  0.8327 &  0.0709  & 0.0709 \\
		0.0354 & 0.0838 &  0.0838 &  0.0068  &  0.0068
		\end{array} \right], \ \ \frac{\cN_\mu'}{10^8} =
		\left[\begin{array}{rrrrr}
		0.0030  &  0.0062  &  0.0062  & 0.0005  &  0.0005 \\
		0.1251 &  0.2708 &  0.2708  & 0.0234  & 0.0234 \\
		0.6091 & 1.3748  & 1.3748 &  0.1158  & 0.1158 \\
		0.5235 & 1.2179  & 1.2179 &  0.1001  & 0.1001 \\
		0.0581 & 0.1381 &  0.1381  & 0.0111  & 0.0111
		\end{array} \right].
		$$
		\normalsize
		By examining the entries of $\cN_\mu$, it follows that the pair of poles $-1\pm i$ are the most sensitive, especially when perturbing the first left measurement value. This can be observed in the $(3,2)$ and $(3,3)$ entries of matrix $\cN_\mu$. Similar conclusion can be drawn when analyzing the entries of matrix  $\cN_\mu'$ (that contains the sensitivities with respect to perturbing the derivative measurements $\bv_i$'). Finally, compute the vector $\bbeta \in \IR^5$ as in (\ref{equ:sen-con1}):
		\begin{equation}
		\bbeta = 10^8 \left[ \begin{matrix}
		1.5005  &  3.4329  &  3.4329   & 0.2868  &  0.2868
		\end{matrix} \right]^T.
		\end{equation}
		We again conclude that the pair of poles $-1\pm i$, is the most sensitive (since $\bbeta_2$ and $\bbeta_3$ are the largest entries). Additionally, note that the highest values computed in this example are in $O(10^8)$, while the ones in Example \ref{ex:4.2} were considerably smaller, i.e., in $O(10^2)$ or $O(10^3)$. 
	\end{example}
	
	

	\subsection{On the choice of the interpolation points and its influence on the sensitivities}

	As shown in the previous sections, the sensitivities $\rho$ and $\eta$ depend on the eigenvectors $\bp$, $\bq$, on the interpolation points $\bmu$, $\blambda$, and on the poles $\bpi$. Moreover, since eigenvectors $\bp$, $\bq$ depend on $\bmu$, $\blambda$, $\bpi$ and also on the residues $\bgamma$, we conclude that sensitivities $\rho$ and $\eta$ are determined by $\bmu$, $\blambda$, $\bpi$ and $\bgamma$. The poles $\bpi$ and residues $\bgamma$ of a given linear system are fixed (system invariants). Hence, the choice of interpolation points will determine sensitivities $\rho$ and $\eta$. Consequently, this will be reflected the robustness of the Loewner model. 
	
	In the Loewner framework the problem of data selection can be split into two sub-problems. The first one stems from the choice of interpolation points, while the second one stems from separating these points into left and right (disjoint) partitions. This is still a complex problem to deal with that is not fully understood. Some progress was made in \cite{KarGosAnt20}, where different types of distributions and separating techniques for the interpolation points were tried.
	
	The sensitivity $\rho$ is directly related to the condition numbers of matrices $\cC_{\mu,\pi}$ and $\cC_{\lambda,\pi}$. Hence, in the case of unstructured perturbation, we can study the problem of choosing interpolation points by exploring the condition number of such generalized Cauchy matrices. The work \cite{Beckermann17} provides a bound for the condition number of the generalized Cauchy matrices for which the denominator part is real. More details are provided in Section \ref{app:1}. However, in many applications, the eigenvalues and interpolation points could be indeed complex numbers. For the sensitivity $\eta$, the problem is more complex because $\eta$ does not depend on the condition number of generalized Cauchy matrices. We illustrate the dependence in $\eta$ by means of numerical examples in Section \ref{sec:4}.

	In what follows we discuss a simplified case. It is assumed that the poles and the interpolation points are well separated into two clusters. The dimension of the Loewner pencil is the same as the order of the original system (denoted with $n$). The Cauchy matrices $\cC_{\mu, \pi}$ and $\cC_{\lambda, \pi}$ are hence square matrices. We also assume that the poles and interpolation points are well separated into two clusters with distance $d$. More precisely, the following relations hold
	\begin{align}\label{equ:distance_bounds}
	\vert \mu_i - \pi_k \vert = O(d),\quad \text{and} \quad\vert \lambda_j - \pi_k \vert = O(d).
	\end{align}
	for $i, j, k \in \left\{1, 2, \ldots, n \right\}$.
	In the sequel we investigate how the distance $d$ affects the sensitivities $\rho$ and $\eta$.
	\begin{lemma}\label{lem:sens_dist}
		Given the unstructured perturbation sensitivity $\rho$ in (\ref{equ:sen-uncon}), the structured perturbation sensitivity $\eta$ in (\ref{equ:sen-con}), and the assumptions made in (\ref{equ:distance_bounds}), the following result hold:
		\begin{align} \label{equ:sen-d}
		\rho(d) = O(d^{4n-4}), \quad \text{and} \quad \eta(d) = O(d^{4n-2}).
		\end{align}
	\end{lemma}
	\begin{proof}
		\noindent
		Given the condition (\ref{equ:distance_bounds}) then the transfer function and the Loewner matrices follow
		\begin{equation}\label{equ:Cauchy_mat_bounds}
		\bH(\mu) = O(d^{-1}), ~~ \bH(\lambda) = O(d^{-1}), ~~
		\Vert \IL  \Vert = O(d^{-2}), ~~ \Vert \sIL  \Vert =  O(d^{-2}).
		\end{equation}

		Given the inverse of Cauchy matrix in equation (\ref{equ:inv_cauchy_def}), it directly follows that the two relations hold 
		\begin{align}\label{equ:eigvec_mat_bounds}
		\Vert \bp \Vert = \Vert \cC_{\mu,\pi}^{-T}  \mathbf{e} \Vert &= O(d^{2n-1}),~~
		\Vert \bq \Vert = \Vert \cC_{\lambda,\pi}^{-T} \mathbf{e} \Vert = O(d^{2n-1}),
		\end{align}
		Using the results stated in (\ref{equ:Cauchy_mat_bounds}), and in  (\ref{equ:eigvec_mat_bounds}), and the definition of $\rho$ in (\ref{equ:sen-uncon}), one can write that
		$$
		\rho(d) =  O(d^{2n-1}) (O(d^{-2}) + O(d^{-2})) O(d^{2n-1}) = O(d^{4n-4}).
		$$
		Similarly, by using the results in equation (\ref{equ:Cauchy_mat_bounds}), and the formula for $\cS_i$, write that
		\begin{align*} 
		\eta(d) &= O(d^{-1}) O(d^n)  O(d^n)  O(d^n)  O(d^{n-1}) = O(d^{4n-2}).
		\end{align*}
	\end{proof}
	
	\vspace{-2mm}
	
	As shown in Lemma \ref{lem:sens_dist}, when $n$ is large, the sensitivities $\rho$ and $\eta$ will increase fast when the distance $d$ is enlarged. So, it is better to choose interpolation points that are not far away from the poles of the system. However, this does not necessarily imply that it is desired to take all measurements close to the poles. For example, if some of the interpolation points are very close to the poles, the values of the measured data could be indeed very large and hence yield an ill-conditioned Loewner pencil.

	\section{Numerical examples}\label{sec:4}
	\label{sec:numerics}
	
	In this section, we provide a numerical study that includes two examples from \cite{Embree19}.
	
	\subsection{Example 1}\label{sec:4.1}
	
	In this section we analyze the first example provided in \cite{Embree19}. Consider a linear system with realization
	$$
	\bA = \left[ \begin{matrix}
	-1.1 & ~1 \\
	~1 & -1.1
	\end{matrix}\right], ~~
	\bB = \left[ \begin{matrix}
	0 \\
	1 
	\end{matrix}\right], ~~
	\bC = \left[ \begin{matrix}
	0 & 1 
	\end{matrix}\right].
	$$
	The poles are given by $\pi_1 = -2.1$ and $\pi_2 = -0.1$.~ 
	Four settings corresponding to the choice of interpolation points are shown in Tab.\;\ref{table:Example 1-1}. Note that the right points are chosen on the real axis in all settings, while the left points are chosen on the imaginary axis for the first three settings. Lastly, for setting 4, all interpolation points are real. 
	
	\begin{table}[h!]
		\scriptsize
		\centering
		\begin{tabular}{|c|c|c|c|c|}
			\hline
			Setting & $\lambda_1$ & $\lambda_2$ & $\mu_1$ & $\mu_2$ \\
			\hline
			1 & 0.00 & 1.00 & 0.00+1.00i & 0.00-1.00i \\
			\hline
			2 & 0.25 & 0.75 & 0.00+2.00i & 0.00-2.00i \\
			\hline
			3 & 0.40 & 0.60 & 0.00+4.00i & 0.00-4.00i \\
			\hline
			4 & 8.00 & 9.00 & 10.00 & 11.00 \\
			\hline
		\end{tabular}
		\caption{\small The four settings for choosing interpolation points.}
		\label{table:Example 1-1}
	\end{table}
	
	The condition number of Cauchy matrices, the sensitivities $\rho_i$ and their bound are shown in Tab.\;\ref{table:Example 1-2}. It is easy to see that the sensitivity $\rho_i$ (corresponding to eigenvalue $\pi_i$ for $i = 1,2$) is strongly related to the condition numbers of the Loewner and generalized Cauchy matrices. An increase in the condition numbers is reflected also in the sensitivity values (see, e.g., for setting 4). This behavior is in accordance to the results provided in Section \ref{sec:sens}.

	
	\begin{table}[h!]
		\scriptsize
		\centering
		\begin{tabular}{|c|c|c|c|c|c|c|}
			\hline
			Setting & cond($\cC_{\mu,\pi}$) & cond($\cC_{\lambda,\pi}$) & $\rho_1$ $\left(\pi_1 = -2.1 \right)$ & $\rho_2$ $\left(\pi_2 = -0.1 \right)$ & $bound(\rho_1)$ & $bound(\rho_2)$ \\
			\hline
			1 & 2.860e+00 & 3.619e+01 & 2.202e+02 & 5.609e-01 & 4.348e+02 & 2.278e+02 \\
			\hline
			2 & 2.740e+00 & 1.958e+01 & 1.049e+02 & 2.191e+00 & 2.253e+02 & 1.180e+02 \\
			\hline
			3 & 4.321e+00 & 3.741e+01 & 2.710e+02 & 1.111e+01 & 6.789e+02 & 3.556e+02 \\
			\hline
			4 & 2.717e+02 & 1.869e+02 & 9.091e+04 & 2.077e+04 & 2.133e+05 & 1.117e+05 \\
			\hline
		\end{tabular}
		\caption{\small Condition numbers and sensitivities $\rho_i$.}
		\label{table:Example 1-2}
	\end{table}

	The pseudospectra computed for each of the four settings are shown in Fig.\;\ref{fig:Example 1-1}. The results match the expectations for sensitivity $\rho$. More precisely, in Setting 4, the slope of the pseudospectrum around the eigenvalues is much smaller than the slope for the other settings. The sensitivity $\rho$ in Setting 4 is also much larger than for other settings. In all the depicted pseudospectra, it is found that the eigenvalue at $-2.1$ is more sensitive than the eigenvalue at $-0.1$. This is also shown for sensitivity $\rho$, i.e., it follows that $\rho_1 \gg \rho_2$.
	
	\begin{figure}[h!]
		\centering
		\includegraphics[width=0.7\linewidth]{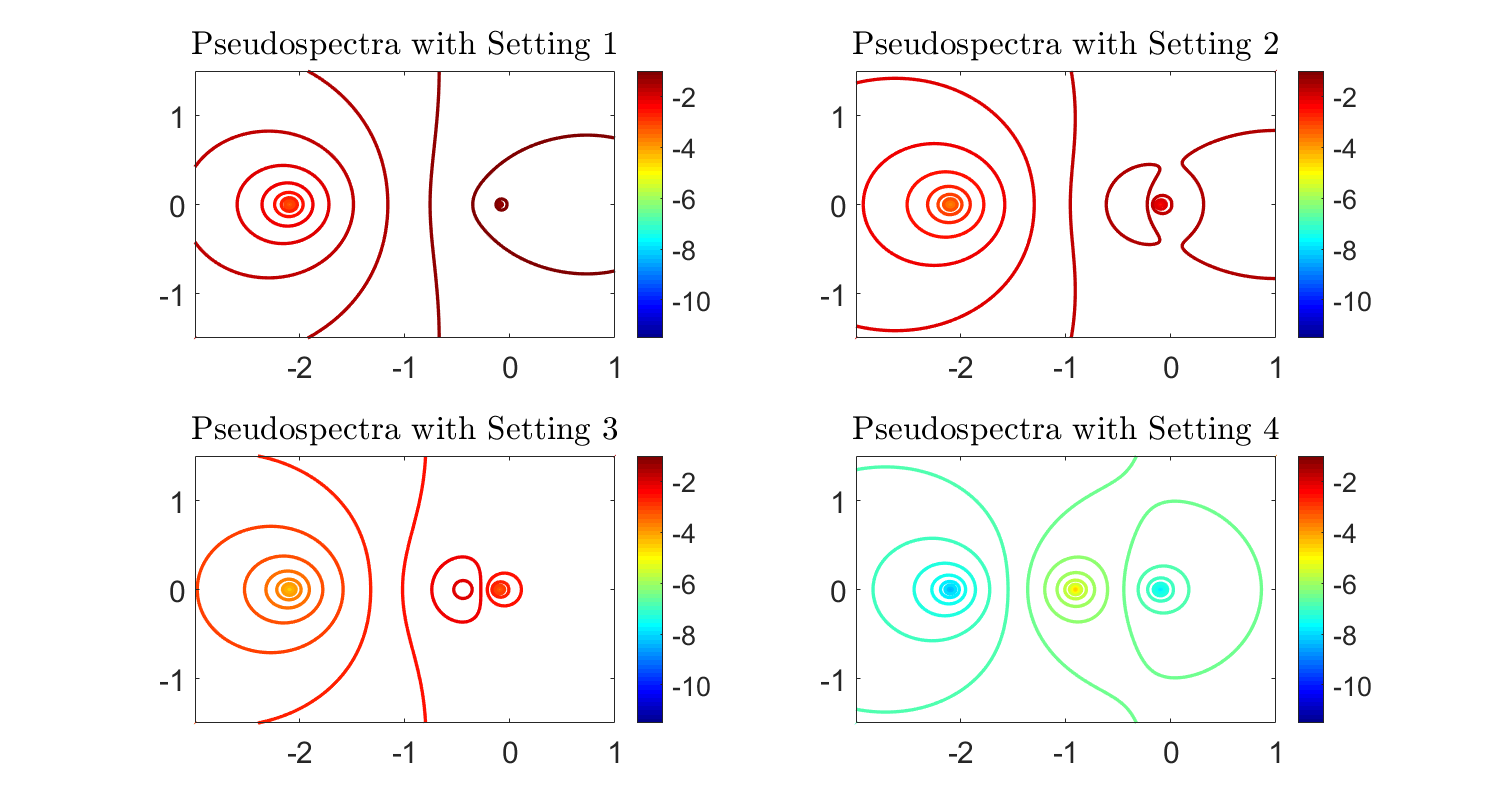}
		\vspace{-4mm}
		\caption{\small Pseudospectra of the Loewner pencils.}
		\label{fig:Example 1-1}	
	\end{figure}
	
	Next, we modify the choice of left and right interpolation points. These will be chosen as shown in Tab.\;\ref{table:Example 1-3}. More precisely, all points will be on the real axis with increasing value for each setting $k \in \{1,2,3\}$. Moreover, the differences $\lambda_2 - \lambda_1$ and $\mu_2-\mu_1$ are kept the same for all three settings (equal to $2$). In Tab.\;\ref{table:Example 1-3}, we illustrate the relationship between sensitivities $\rho$, $\eta$ and $d$, i.e., the distance from the eigenvalue cluster to the interpolation points cluster. By following the results in (\ref{equ:sen-d}), we can indeed show that $\eta = O(d^6)$ and $\rho = O(d^4)$. The results presented in Tab.\;\ref{table:Example 1-3} indeed match the theoretical prediction.
	
	\begin{table}[h!]
		\scriptsize
		\centering
		\begin{tabular}{|c|c|c|c|c|c|c|c|}
			\hline
			$\lambda_1$ & $\lambda_2$ & $\mu_1$ & $\mu_2$ & $\rho_1$ ($p_1$=-0.1) & $\rho_2$ ($p_2$=-2.1) & $\eta_1$ ($p_1$=-0.1) & $\eta_2$ ($p_2$=-2.1) \\
			\hline
			0.00 & 2.00 & 1.00 & 3.00 & 2.881e+00 & 1.295e+03 & 2.848e+00 & 2.758e+02 \\
			\hline
			10.00 & 12.00 & 11.00 & 13.00 & 1.124e+04 & 4.551e+04 & 1.144e+06 & 1.855e+06 \\
			\hline
			100.00 & 102.00 & 101.00 & 103.00 & 6.415e+07 & 1.797e+08 & 4.220e+11 & 4.475e+11 \\
			\hline
		\end{tabular}
		\caption{Sensitivities for different choices of interpolation points.}
		\label{table:Example 1-3}
	\end{table}

	\normalsize

	\subsection{Example 2}\label{sec:4.2}
	
	In this section we analyze the second example provided in \cite{Embree19}. Consider a linear system with realization given by matrices
	$$
	\bA = \mbox{diag}(-1,-2,\cdots, -10), ~~
	\bB = \left[ \begin{matrix}
	1, 1, \cdots, 1
	\end{matrix}\right]^T, ~~
	\bC = \left[ \begin{matrix}
	1, 1, \cdots, 1
	\end{matrix}\right].
	$$
	The poles of the system are $\{-1, -2, \ldots, -10\}$ corresponding to residues $\{ 1, 1, \ldots, 1\}$. As given in Tab.\;\ref{table:Example 2-1}, we choose two settings for the interpolations points. Note that both settings share the same interpolation points, chosen on the negative real axis inside the interval $(-11,0)$. The difference between them is given by ordering  and by separating into the left/right partitions. Note that the left/right interpolation points corresponding to Setting 1 are interlaced, while the left/right interpolation points in Setting 2 are completely separated (half-half).
	
	Additionally, the sensitivities $\rho$ and $\eta$  corresponding to each of the two separation settings, are also listed in Tab.\;\ref{table:Example 2-1}.  An interesting numerical result can be observed: sensitivity $\rho$ in Setting 2 is much larger than that in Setting 1 which implies the model in Setting 2 is more ill-conditioned than the model in Setting 1. However, we also note that the sensitivity $\eta$ in Setting 2 is comparable to that of Setting 1. This means that the Loewner model in Setting 1 is as robust as that in Setting 2, with respect to perturbing the data. 
	
	\begin{table}[h!]
		\centering
		\scriptsize
		\begin{tabular}{|c|c|c|c|c|}
			\hline
			$\lambda$ & $\mu$ & $\pi$ & $\rho$ & $\eta$ \\
			\hline
			-10.25 & -9.75 & -10.000 & 2.205e+01 & 2.098e-01 \\
			\hline
			-9.25 & -8.75 & -9.000 & 1.947e+01 & 1.836e-01 \\
			\hline
			-8.25 & -7.75 & -8.000 & 1.812e+01 & 1.711e-01 \\
			\hline
			-7.25 & -6.75 & -7.000 & 1.697e+01 & 1.647e-01 \\
			\hline
			-6.25 & -5.75 & -6.000 & 1.590e+01 & 1.619e-01 \\
			\hline
			-5.25 & -4.75 & -5.000 & 1.487e+01 & 1.619e-01 \\
			\hline
			-4.25 & -3.75 & -4.000 & 1.387e+01 & 1.647e-01 \\
			\hline
			-3.25 & -2.75 & -3.000 & 1.292e+01 & 1.711e-01 \\
			\hline
			-2.25 & -1.75 & -2.000 & 1.208e+01 & 1.836e-01 \\
			\hline
			-1.25 & -0.75 & -1.000 & 1.185e+01 & 2.098e-01 \\
			\hline
		\end{tabular}
		\scriptsize
		\begin{tabular}{|c|c|c|c|c|}
			\hline
			$\lambda$ & $\mu$ & $\pi$ & $\rho$ & $\eta$ \\
			\hline
			-5.25 & -10.25 & -10.000 & 5.857e+06 & 2.098e-01 \\
			\hline
			-4.75 & -9.75 & -9.000 & 9.429e+06 & 1.836e-01 \\
			\hline
			-4.25 & -9.25 & -8.000 & 6.653e+06 & 1.711e-01 \\
			\hline
			-3.75 & -8.75 & -7.000 & 2.704e+06 & 1.647e-01 \\
			\hline
			-3.25 & -8.25 & -6.000 & 1.578e+06 & 1.619e-01 \\
			\hline
			-2.75 & -7.75 & -5.000 & 1.447e+06 & 1.619e-01 \\
			\hline
			-2.25 & -7.25 & -4.000 & 2.082e+06 & 1.647e-01 \\
			\hline
			-1.75 & -6.75 & -3.000 & 4.285e+06 & 1.711e-01 \\
			\hline
			-1.25 & -6.25 & -2.000 & 5.042e+06 & 1.836e-01 \\
			\hline
			-0.75 & -5.75 & -1.000 & 2.571e+06 & 2.098e-01 \\
			\hline
		\end{tabular}\\
		\caption{Poles and sensitivities of Loewner models of Setting 1 and 2.}
		\label{table:Example 2-1}
	\end{table}

	Next, we generate 1000 Loewner models for measurements corrupted by Gaussian noise with $\sigma = 0.3$. For each trial, we display in Fig.\;\ref{fig:Example 2-2} the poles for both noisy and noiseless case in each of the two settings. The results depicted there show that the Loewner models constructed from noisy data by Setting 1 have a similar distribution of eigenvalues compared with the models in Setting 2. This numerical example illustrates that it is indeed meaningful to define the sensitivity with respect to both unstructured and structured perturbations.
	
	\begin{figure}[h!]
		\centering
		\includegraphics[width=0.8\linewidth]{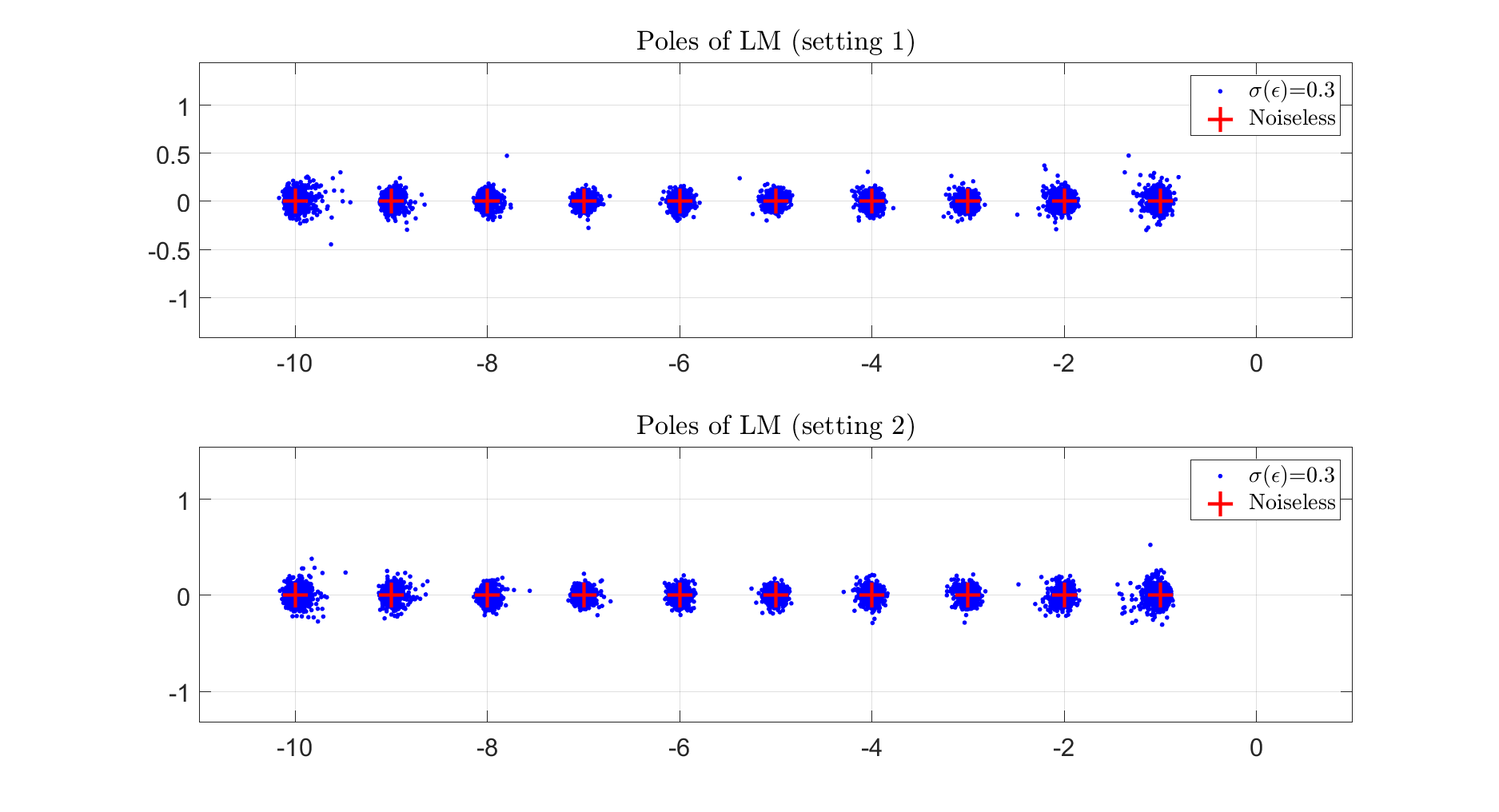}
		\vspace{-4mm}
		\caption
		{The distribution of poles of Loewner system models constructed from noisy data.}
		\label{fig:Example 2-2}
	\end{figure}
	
	As shown in Tab.\;\ref{table:Example 2-2}, the condition numbers for the Cauchy matrices defined for Setting 2 are considerably larger than those of the matrices defined for Setting 1 (approximately 5 to 6 orders of magnitude). Next, compute the $\ell_2$-norm of the sensitivities $\boldsymbol{\rho}$. This quantity takes into account the unstructured sensitivities with respect to all the poles $\pi_i$. Additionally, compute the upper bound on $\vert \boldsymbol{\rho} \vert_2$ provided in (\ref{equ:bound2}). Note that this bound is much tighter for Setting 1.

	\begin{table}[h!]
		\centering
		\scriptsize
		\begin{tabular}{|c|c|c|c|c|}
			\hline
			Setting  & $\kappa(\cC_{\mu,\pi})$ & $\kappa(\lambda, \pi)$ & $\Vert \boldsymbol{\rho} \Vert_2$& 
			$bound(\Vert \boldsymbol{\rho} \Vert_2)$
			\\
			\hline 
			1  & 1.217e+00 & 1.217e+00 & 5.100e+01 &  7.385e+01\\
			\hline
			2  & 1.771e+06 & 1.771e+06 & 1.530e+07& 1.563e+14\\
			\hline
		\end{tabular}
		\caption{Condition numbers of different matrices and the coefficient $\Vert \boldsymbol{\rho} \Vert_2$.}
		\label{table:Example 2-2}
	\end{table} 
	
	The pseudospectra in Fig. \ref{fig:Example 2-1} also shows that the eigenvalues in Setting 2 are more sensitive to noise than those in Setting 1. This example illustrates how separating the interpolation points (into left and right subsets) can greatly affect the condition numbers of the Loewner pencil and generalized Cauchy matrices. It is to be concluded that interlaced data separation seems to be more advantageous than splitting the interpolation data in half, in order to avoid ill-conditioning of the Loewner model.

	\begin{figure}[h!]
		\centering
		\includegraphics[width=1\linewidth]{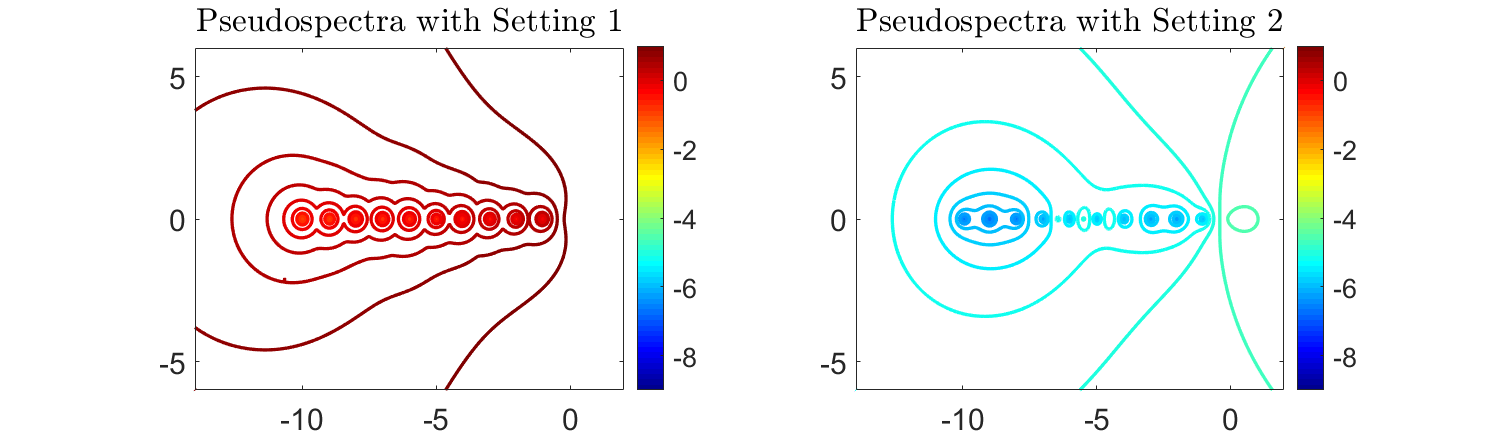}
		\caption{Pseudospectra of the Loewner pencils.}
		\label{fig:Example 2-1}
	\end{figure}
	
	In Fig. \ref{fig:EX2_eta_1}, we display a heat map of sensitivity values $\eta_{(i,j)}$ for the two settings described before. Hence the entries of matrices $\cN_\mu$ and $\cN_\lambda$ are displayed. This shows that the sensitivity values are large whenever the measurements are closer to the poles. This means that measurements that are close to poles will greatly affect the perturbation with respect to those particular poles. 
	
	\begin{figure}	
		\hspace{-14mm}
		\includegraphics[width=1.2\linewidth]{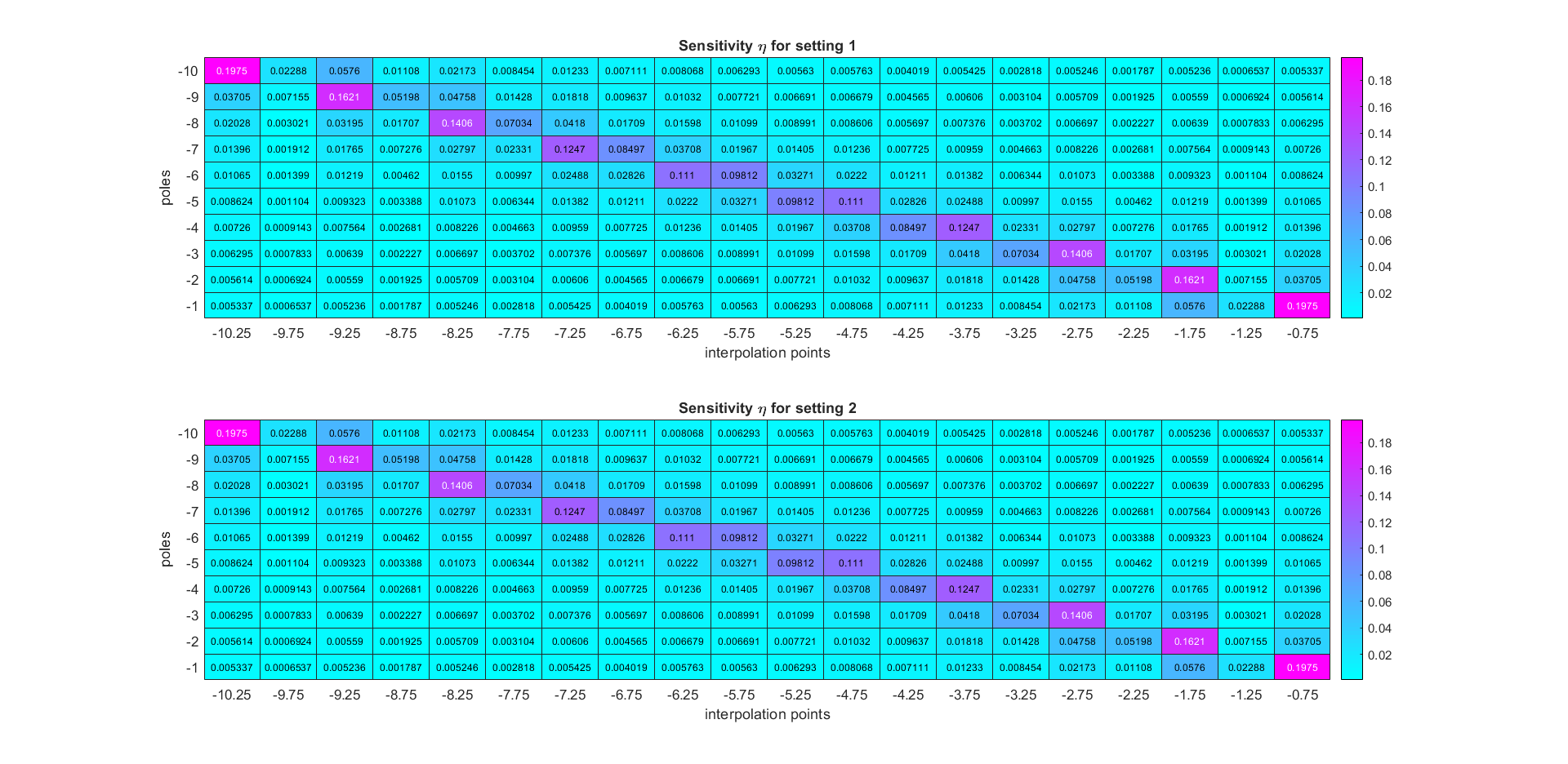}
		\vspace{-10mm}
		\caption{Structured sensitivity $\eta_{(i,j)}$ for the first two settings; interpolation points are  on the x-axis, while the poles are on the y-axis.}
		\label{fig:EX2_eta_1}
	\end{figure}

	\subsubsection{Bounds on the singular values of Cauchy and Loewner matrices} \label{app:1}
	
	In \cite{Beckermann17}, a bound on the decay of the singular values for matrices with displacement structure is provided. Such matrices satisfy Sylvester equations. Indeed, the Cauchy and Loewner matrices are two types of matrices with displacement structure. For the Cauchy matrix $\cC_{x,y} \in \mathbb{C}^{m \times n}$ ($m \ge n$) where the entries of $x$ are located in the interval $\left[a, b\right]$ and the entries of $y$ are located in the interval $\left[c, d\right]$, the following bound holds (provided that the two intervals are disjoint)
	\begin{equation} \label{equ:bound_cauchy}\small
	\sigma_{j+k} \left(\cC_{x,y}\right) \le 
	4 \left[ \mathrm{exp}\left(\frac{\pi^2}{4\mu\left(1/\sqrt{\gamma}\right)}\right) \right]^{-2k}
	\sigma_{j} \left(\cC_{x,y}\right), ~~~~~ 1\le j+k \le n
	\end{equation}
	where $\gamma = \vert \left( c-a \right)\left( d-b \right) / 
	\left( \left( c-b \right) \left( d-a \right) \right) \vert$ is the absolute value of the cross-ratio of $a,b,c,d$ and $\mu$ is the Grotzsch ring function.
	Thus, an upper bound for the decay of singular values is obtained which is independent from the numerators of Cauchy matrix entries. In our work, numerators of entries are residues of the poles of the system. Tab.\;\ref{table:Example 1-4} shows the experimental result of the bound on example \ref{sec:4.1}. Fig. \ref{fig:Example4_2_1} shows the experimental results of the bound on example \ref{sec:4.2}.
	\begin{table}[h]
		\scriptsize
		\centering
		\begin{tabular}{|c|c|c|c|c|c|c|c|}
			\hline
			$\mu_1$ & $\mu_2$ & $\lambda_1$ & $\lambda_2$ & $\kappa(\cC_{\tiny\mbox{L}})$ & bound($\kappa(\cC_{\tiny\mbox{L}})$) & $\kappa(\cC_{\tiny\mbox{R}})$ & bound($\kappa(\cC_{\tiny\mbox{R}})$) \\
			\hline
			1.0 & 3.0 & 0.0 & 2.0 & 1.439e+01 & 7.443e+00 & 4.541e+01 & 1.732e+00 \\
			\hline
			11.0 & 13.0 & 10.0 & 12.0 & 1.737e+02 & 1.696e+02 & 1.485e+02 & 1.444e+02 \\
			\hline
			101.0 & 103.0 & 100.0 & 102.0 & 1.063e+04 & 1.063e+04 & 1.043e+04 & 1.042e+04 \\
			\hline
		\end{tabular}
		\caption{Condition numbers and sensitivity $\rho$}
		\label{table:Example 1-4}
	\end{table}\\[-4mm]

	A similar bound holds for Loewner matrices. Let $\IL_{x,y}\in \mathbb{C}^{m \times n}$ be a Loewner matrix, where $m \ge n$, the following bound for the decay of singular values was given in \cite{Beckermann17}:
	\begin{equation} \label{equ:bound_loewner}\small
	\sigma_{j+2k} \left(\IL_{x,y}\right) \le 
	4 \left[ \mathrm{exp}\left(\frac{\pi^2}{4\mu\left(1/\sqrt{\gamma}\right)}\right) \right]^{-2k}
	\sigma_{j} \left(\IL_{x,y}\right), ~~ 1\le j+2k \le n.
	\end{equation}
	
	The figure below provide results of the experiment on the bound of example \ref{sec:4.2}.
	
	\begin{figure}[h!]
		\centering
		\hspace{-8mm}
		\includegraphics[width=0.54\linewidth]{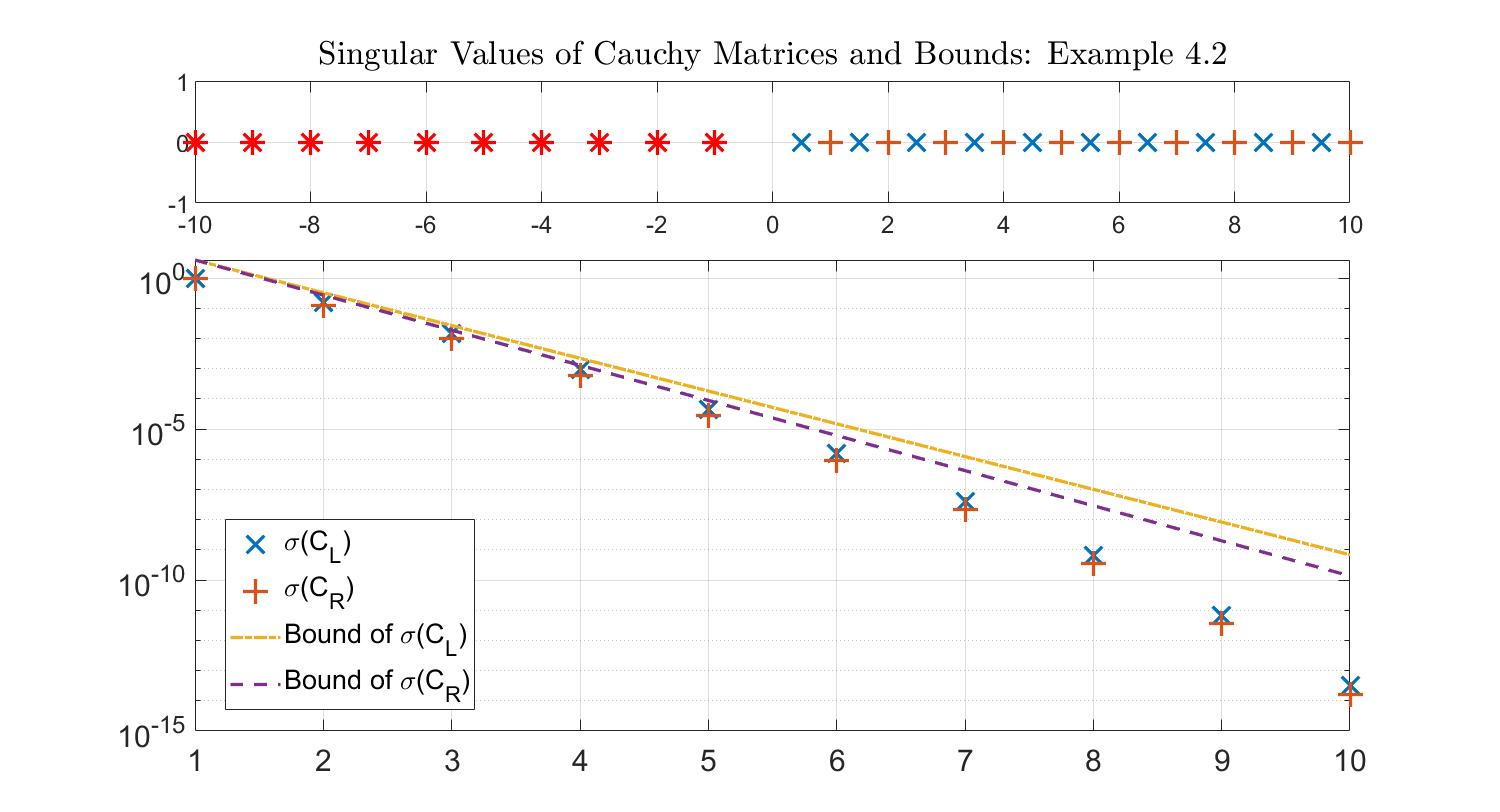} \hspace{-8mm}
		\includegraphics[width=0.54\linewidth]{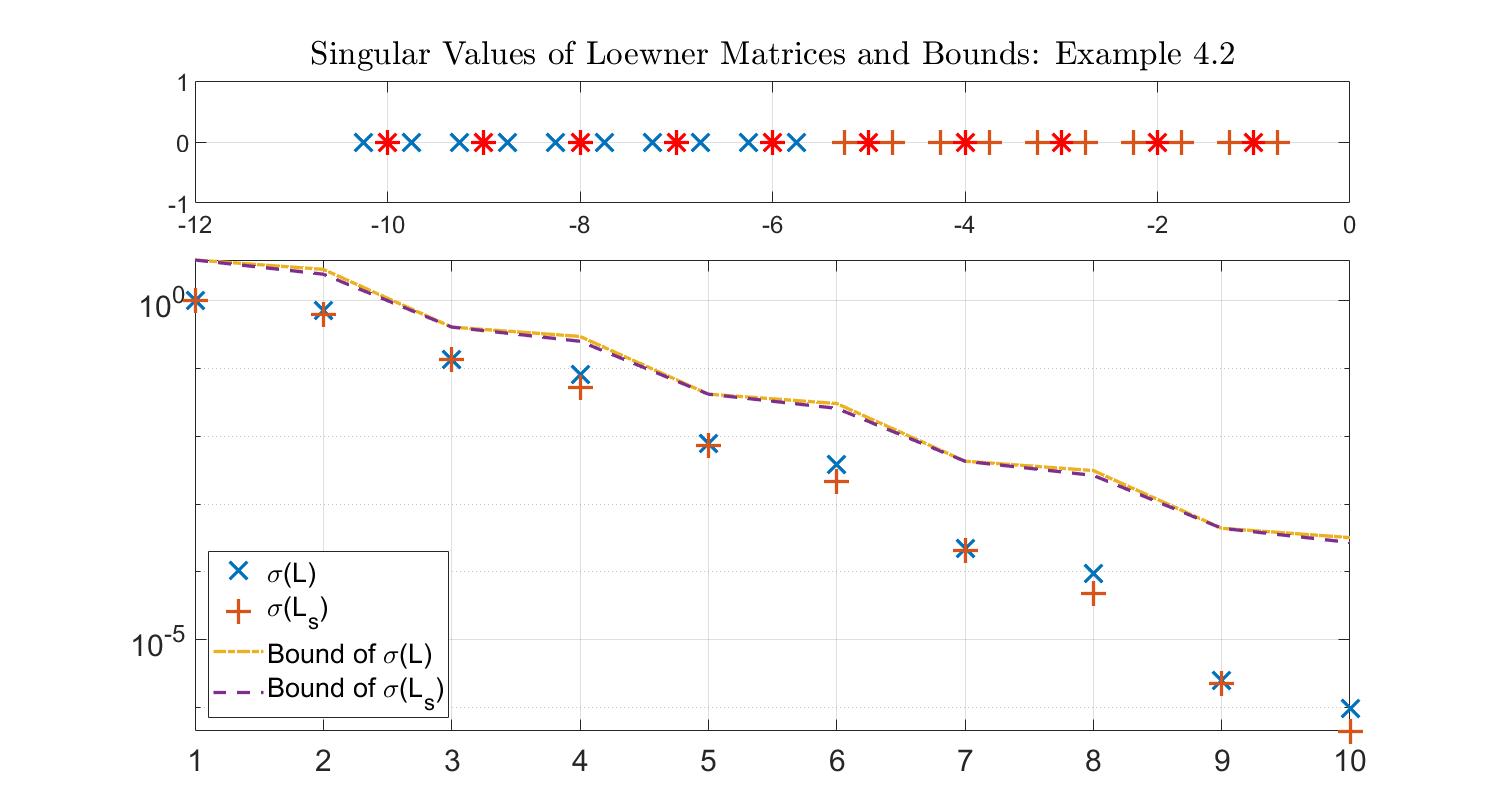}
		\vspace{-2mm}
		\caption{Singular value decay of Cauchy and Loewner matrices with the corresponding bounds for example \ref{sec:4.2}. The red asterisks represent poles of the system, blue crosses represent left interpolation points $\mu$, and orange plus signs represent right interpolation points $\lambda$ (for two different settings provided in Section \ref{sec:4.2}).}
		\label{fig:Example4_2_1}
	\end{figure}

	\section{Conclusion}
	\label{sec:conc}
	
	In this paper, we have presented a factorization of the Loewner pencil which yields an explicit generalized eigenvalue decomposition of the pencil. Based on this decomposition and on results from perturbation theory of eigenvalues corresponding to a matrix pencil, the sensitivities of eigenvalues of the Loewner pencil are defined and analyzed. This is done with respect to both unstructured and structured types of perturbations.

	It was found that the sensitivity of eigenvalues with unstructured perturbation is related to the condition numbers of the generalized Cauchy matrices that appear in the factorization of the Loewner pencil. This can indeed represent a useful tool for eigenvalue sensitivity analysis, together with the pseudospectrum.

	The sensitivity of eigenvalues with structured perturbations is an important tool to be used in the case of noisy data. The relationship between the perturbation of eigenvalues and the perturbation in the data is explored. One issue dealt with in this work was to analyze the robustness of the Loewner model with respect to perturbed data. Using eigenvalue sensitivity analysis, we showed how the choice of interpolation points affects the Loewner model. 
	

	In this work, we have explored some meaningful developments of the eigenvalue sensitivity analysis for the Loewner pencil. Nevertheless, there are still some open problems to be dealt with. For example, one needs to further investigate the problem of choosing the interpolation points. Additionally, the study of sensitivity for eigenvalues with multiplicity could also represent a topic of further research. Finally, the sensitivity computation procedure needs to be expanded for the case of redundant data. Consequently, the sensitivity analysis can also be used for reduced-order models constructed within the Loewner framework.
	
	\small
	
\bibliographystyle{plainurl}
	\bibliography{bibliography}
	
	\normalsize

\end{document}

%% file: mymacros.tex
\newfont{\Bb}{msbm10 scaled\magstep0}

\def\IR{\mathbb R}
\def\IC{\mathbb C}

\def\II{\mathbb I}

\def\IL{\mathbb L}

\def\IV{\mathbb V}
\def\IW{\mathbb W}

\newtheorem{lemma}{Lemma}[section]
\newtheorem{definition}{Definition}[section]
\newtheorem{proposition}{Proposition}[section]

\newtheorem{example}{Example}[section]
\newtheorem{remark}{Remark}[section]

\newcommand{\blem}{\begin{lemma}}
	\newcommand{\elem}{\end{lemma}}
\newcommand{\bea}{\begin{eqnarray*}}
	\newcommand{\eea}{\end{eqnarray*}}

\newcommand{\bA}{{\mathbf A}}
\newcommand{\bB}{{\mathbf B}}
\newcommand{\bC}{{\mathbf C}}
\newcommand{\bD}{{\mathbf D}}
\newcommand{\bE}{{\mathbf E}}

\newcommand{\bJ}{{\mathbf J}}

\newcommand{\bL}{{\mathbf L}}
\newcommand{\bM}{{\mathbf M}}

\newcommand{\bI}{{\mathbf I}}
\newcommand{\bH}{{\mathbf H}}

\newcommand{\bW}{{\mathbf W}}
\newcommand{\bR}{{\mathbf R}}

\newcommand{\bX}{{\mathbf X}}
\newcommand{\bx}{{\mathbf x}}
\newcommand{\by}{{\mathbf y}}
\newcommand{\bu}{{\mathbf u}}
\newcommand{\bV}{{\mathbf V}}

\newcommand{\bc}{{\mathbf c}}
\newcommand{\bfe}{{\mathbf e}}

\newcommand{\bn}{{\mathbf n}}
\newcommand{\bd}{{\mathbf d}}
\newcommand{\bb}{{\mathbf b}}

\newcommand{\bp}{{\mathbf p}}
\newcommand{\bq}{{\mathbf q}}
\newcommand{\bh}{{\mathbf h}}

\newcommand{\bv}{{\mathbf v}}
\newcommand{\bw}{{\mathbf w}}

\newcommand{\br}{{\mathbf r}}

\newcommand{\bfz}{{\mathbf 0}}
\newcommand{\cC}{ {\cal C} }

\newcommand{\cK}{ {\cal K} }

\newcommand{\cH}{ {\cal H} }

\newcommand{\cN}{ {\cal N} }

\newcommand{\cV}{ {\cal V} }

\newcommand{\cS}{ {\cal S} }
\newcommand{\cT}{ {\cal T} }

\newcommand{\cW}{ {\cal W} }
\newcommand{\Si}{ \boldsymbol{\Sigma} }

\newcommand{\bPi}{ \boldsymbol{\Pi} }
\newcommand{\bpi}{ \boldsymbol{\pi} }
\newcommand{\bmu}{ \boldsymbol{\mu} }
\newcommand{\blambda}{ \boldsymbol{\lambda} }
\newcommand{\bTheta}{ \boldsymbol{\Theta} }

\newcommand{\bPhi}{ \boldsymbol{\Phi} }
\newcommand{\bepsilon}{ \boldsymbol{\epsilon} }
\newcommand{\bbeta}{ \boldsymbol{\eta} }

\newcommand{\bPsi}{ \boldsymbol{\Psi} }

\newcommand{\bgamma}{\boldsymbol{\gamma}}

\newcommand{\bell}{\boldsymbol{\ell}}
\newcommand{\bDelta}{\boldsymbol{\Delta}}
\newcommand{\bLambda}{\boldsymbol{\Lambda}}

\newcommand{\bGamma}{ \boldsymbol{\Gamma} }

\newcommand{\sIL}{{\mathbb L}_s}

\newlength{\boxedparwidth}  
{
	\begin{center}
		\begin{tabular}{|@{\hspace{.15in}}c@{\hspace{.15in}}|}
			\hline 
			\begin{minipage}[t]{\boxedparwidth}
				\setlength{\parindent}{.25in}}%
			{\end{minipage} \\[1mm] \\[1mm] \hline 
		\end{tabular} 
	\end{center} 
}